\documentclass[english]{amsart}

\usepackage{comment}

\usepackage{amsfonts, amssymb, amsbsy, amsmath, mathrsfs, amsthm,enumerate, bm}

\usepackage{graphicx}
\usepackage{url}
\usepackage{xcolor}
\usepackage{hyperref}
\usepackage{cleveref}

\newcommand\dmu{\, d \mu}
\newcommand\dt{\, d t}
\newcommand\dnu{\, d \nu}
\newcommand\ds{\, d s}

\newcommand\R{\mathbb R}

\usepackage{mathrsfs}

\theoremstyle{definition}
\newtheorem{definition}{Definition}[section]

\newtheorem{remark}[definition]{Remark} 
\theoremstyle{plain}
\newtheorem{theorem}[definition]{Theorem}
\newtheorem{lemma}[definition]{Lemma}
\newtheorem{corollary}[definition]{Corollary}
\newtheorem{proposition}[definition]{Proposition}

\numberwithin{equation}{section}

\crefname{theorem}{Theorem}{Theorems}
\Crefname{theorem}{Theorem}{Theorems}
\crefname{lemma}{Lemma}{Lemmas}
\Crefname{lemma}{Lemma}{Lemmas}
\crefname{proposition}{Proposition}{Propositions}
\Crefname{proposition}{Proposition}{Propositions}
\crefname{corollary}{Corollary}{Corollaries}
\Crefname{corollary}{Corollary}{Corollaries}
\crefname{definition}{Definition}{Definitions}
\Crefname{definition}{Definition}{Definitions}
\crefname{remark}{Remark}{Remarks}
\Crefname{remark}{Remark}{Remarks}
\crefname{example}{Example}{Examples}
\Crefname{example}{Example}{Examples}

\def\Xint#1{\mathchoice
   {\XXint\displaystyle\textstyle{#1}}%
   {\XXint\textstyle\scriptstyle{#1}}%
   {\XXint\scriptstyle\scriptscriptstyle{#1}}%
   {\XXint\scriptscriptstyle\scriptscriptstyle{#1}}%
   \!\int}
\def\XXint#1#2#3{{\setbox0=\hbox{$#1{#2#3}{\int}$}
     \vcenter{\hbox{$#2#3$}}\kern-.5\wd0}}

\def\dashint{\Xint-}

\def\YYint#1#2#3{{\setbox0=\hbox{$#1{#2#3}{\iint}$}
    \vcenter{\hbox{$#2#3$}}\kern-.51\wd0}}

\renewcommand{\div}{\operatorname{div}}

\DeclareMathOperator*{\esssup}{ess\,sup\,}
\DeclareMathOperator*{\essinf}{ess\,inf\,}

\DeclareMathOperator{\diam}{diam}

\DeclareMathOperator\supp{supp}

\usepackage{natbib}
\setcitestyle{numbers,open={[},close={]},comma}

\title[Characterization of intrinsic parabolic Harnack inequalities]{Characterization of intrinsic parabolic Harnack inequalities for doubly nonlinear equations of slow diffusion type
}
\author{Theo Elenius}
\address{Department of Mathematics, Aalto University, P.O. BOX 11100, 00076 Aalto, Finland}
\email{theo.elenius@aalto.fi}

\begin{document}
\bibliographystyle{plainnat}
\begin{abstract}
We characterize metric measure spaces that satisfy intrinsic parabolic Harnack inequalities for doubly nonlinear equations of slow diffusion type. The considered parabolic Harnack inequalities are intrinsic in the sense that the waiting time depends on the solution. Our main results shows that in a metric measure space with a doubling measure, such an intrinsic parabolic Harnack inequality holds if and only if a Poincaré inequality holds. We thus show that estimates in scale and function dependent intrinsic geometries can characterize scale and location invariant functional inequalities. 
\end{abstract}


\maketitle

\section{Introduction}
Characterizations of parabolic Harnack inequalities for linear equations have been extensively studied, beginning with the work of Grigor'yan \cite{Grigoryan1991} and Saloff-Coste \cite{Saloff-Coste1992,Saloff-Coste2002}, who independently observed that a parabolic Harnack inequality for the heat equation is equivalent to the conjunction of volume doubling and a $(1,2)$-Poincaré inequality on Riemannian manifolds. Corresponding characterizations were shown, for example, for locally finite graphs \cite{Delmotte} and locally compact metric spaces with a Dirichlet form, see \cite{Sturm} and \cite{, BarlowBassKumagai}. 

In recent joint work with Kinnunen \cite{EleniusKinnunen} we have developed a nonlinear analytic approach to the characterization of parabolic Harnack inequalities and have shown that the conjunction of volume doubling and a $(1,p)$-Poincaré inequality is equivalent to a parabolic Harnack inequality for the doubly nonlinear equation
\begin{equation}
    \label{eq: trudinger}
     \partial_t (\lvert u\rvert^{p-2} u)-\div (\lvert \nabla u\rvert^{p-2}\nabla u )=0, \quad 1<p<\infty, 
\end{equation}

In this paper, we consider the characterization of parabolic Harnack estimates for the doubly nonlinear equation
\begin{equation}
    \label{eq: PDE}
     \partial_t (\lvert u\rvert^{q-1} u)-\div (\lvert \nabla u\rvert^{p-2}\nabla u )=0, \quad 1<p<\infty, \, q>0,
\end{equation}
in the slow diffusion regime $q+1<p$. The limiting case $q+1=p$ corresponds to \eqref{eq: trudinger}. Examples of equations covered by \eqref{eq: PDE} include the parabolic $p$-Laplace, i.e. the gradient flow of the $p$-Sobolev energy, 
corresponding to $q=1$, and the porous medium equation
\begin{equation*}
    \partial_t v - \Delta (v^m)=0, \quad 0<m<\infty
\end{equation*}
for the choices $p=2$, $q=\frac{1}{m}$ and change of variables $v^m=u$. Recently, the equation \eqref{eq: PDE} has garnered interest in the context of Riemannian manifolds. Upper bounds for solutions and their relation to Sobolev and Faber-Krahn type inequalities have been investigated in \cite{GrigoryanSurig2025,GrigoryanSunSurig,sürig2026upperbounds}, finite propagation speed has been shown in \cite{GrigoryanSurig2024, GrigoryanSurig2024Pisa} and an existence result has been given in \cite{sürig2026existence}.

The behavior of solutions in the slow diffusion regime is markedly different compared to \eqref{eq: trudinger}; disturbances propagate at finite speed and a uniform parabolic Harnack inequality does not hold. One of the key structural differences causing this discrepancy is the scaling of the equation. Note that if $u$ is a solution to \eqref{eq: PDE} and $a>0$, then $au$ is not a solution to \eqref{eq: PDE} for $q+1\neq p$. Instead,  the function $w(x,t)=au(x,a^{p-1-q}t)$ is a solution.

Intrinsic scaling methods have been extensively developed to study regularity for equations of the form of \eqref{eq: PDE}, see the monographs \cite{DiBenedetto1993,DiBenedettoGIanazzaVespri2012,Urbano}. The idea is to establish homogeneous estimates in intrinsic cylinders, which are cylinders whose time length depend on the operator and  solution itself, in a way which respects the natural scaling of the equation. This idea has been used to obtain, for instance, Harnack estimates, Hölder regularity and gradient higher integrability.

The intrinsic parabolic Harnack  inequality asserts that there exists constants $c_1$ and $c_2$ such that whenever $u$ is a nonnegative solution of \eqref{eq: PDE} in $Q=B(x_0,4r)\times(t_0-c_1u(x_0,t_0)^{q+1-p} (4r)^p,t_0+c_1u(x_0,t_0)^{q+1-p}(4r)^p)$ then
\begin{equation}
    \label{eq: intrinsic PHI}
    u(x_0,t_0)\leq c_2\essinf_{Q^+}u 
\end{equation}
where $Q^+=B(x_0,r)\times(t_0+c_1/2 u(x_0,t_0)^{q+1-p}r^p,t_0+c_1 u(x_0,t_0)^{q+1-p} r^p)$.

Our main result shows that in a metric measure space with a doubling measure, the intrinsic parabolic Harnack inequality \eqref{eq: intrinsic PHI} holds if and only if a $(1,p)$-Poincaré inequality holds. The main novelty of our work is to show that a Harnack inequality in a solution and location dependent geometry suffices to characterize a scale and location invariant Poincaré inequality.  More precisely, we prove the following theorem.
\begin{theorem}
\label{thm: main thm}
Let $1<p<\infty$ and $0<q<\infty$ with $q+1<p$. Let $(X,d,\mu)$ be a metric measure space such that $X$ is geodesic and complete, $\mu$ is doubling and for $q\neq  1$ assume that the Rellich--Kondrachov property in \Cref{def: rellich kondrachov property} holds. Then the following conditions are equivalent:
    \begin{itemize}
        \item[(1)] A $(1,p)$-Poincar\'e inequality holds in $(X,d,\mu)$, see \Cref{def: poincare}.
        \item[(2)] The intrinsic parabolic Harnack inequality holds in $(X,d,\mu)$, see \Cref{{def: PHI}}.
    \end{itemize}
    Moreover, the constants appearing in $(1)$ and $(2)$ are quantitative in terms of those in $(2)$ and $(1)$, respectively, and $p$, $q$, $c_\mu$.
\end{theorem}
The overall strategy of the proof of \Cref{thm: main thm} is an adaptation of the argument of \cite{EleniusKinnunen} to the intrinsic setting. We begin with the proof of the implication $(1)\implies (2)$, which is presented in \Cref{section: VD + PI implies PHI}. The proof is based on an expansion of positivity argument, which is adapted to include the intrinsic geometry. For the corresponding result in the Euclidean setting we refer to \cite{DiBenedettoGianazzaVespri2008} and \cite{BogeleinHeranSchatzlerSinger}. The intrinsic parabolic Harnack inequality for $q=1$ has previously been shown in the metric setting in \cite{Heran}. 

To show that the intrinsic parabolic Harnack inequality implies a $(1,p)$-Poincaré inequality, an existence result for a Cauchy problem related to \eqref{eq: PDE} is needed. A proof of such an existence result is given in \Cref{section: existence}, using the minimizing movements scheme. This is the only place we apply the Rellich--Kondrachov property which is assumed in \Cref{thm: main thm}. It appears in relation to the nonlinearity of the time derivative whenever $q\neq 1$. We point out that the Rellich--Kondrachov property holds independently of volume doubling and a $(1,p)$-Poincaré inequalities, for instance, on Riemannian manifolds.

The proof of the implication $(2)\implies(1)$ is given in \Cref{section: PHI implies Poincare}. The proof proceeds by first showing that the intrinsic Harnack inequality implies a Poincaré-type inequality for functions with zero boundary values. To show this, we adapt the argument of \cite[Proposition 7.1]{EleniusKinnunen} to take into account the intrinsic scaling of the parabolic Harnack inequality. The following step is to apply a result of \cite{EleniusKinnunen} which states that a Poincaré-type inequality for functions with zero boundary values, the elliptic Harnack inequality for $p$-harmonic functions and volume doubling imply a $(1,p)$-Poincaré inequality.

We point out that \Cref{thm: main thm} assumes that the measure is doubling, whereas the characterization in \cite{EleniusKinnunen} and the classical linear theory show that a parabolic Harnack inequality implies volume doubling. To the best of our knowledge, all proofs of the doubling condition starting from the parabolic Harnack inequality rely in some way on iterating the parabolic Harnack inequality at small scales to move forward in time and eventually be able to apply the parabolic Harnack inequality at a larger scale.  Such arguments are not immediately available for the intrinsic parabolic Harnack inequality, since the solution may decay during the iteration, while the the backward time interval required to apply the parabolic Harnack inequality depends on $u^{q+1-p}$.  It is thus an interesting open problem to determine whether the validity of the intrinsic parabolic Harnack inequality implies volume doubling. 

There are several interesting consequences of \Cref{thm: main thm}. In particular, \Cref{thm: main thm} shows that if the intrinsic parabolic Harnack inequality holds for some $q<p-1$ then it holds for every $q<p-1$. Together with \cite[Theorem 1.1]{EleniusKinnunen} we also obtain the equivalence of the intrinsic parabolic Harnack inequality for any $q<p-1$ with the parabolic Harnack inequality for $q=p-1$. This shows a somewhat unexpected relation between parabolic Harnack estimates for different equations. The stability of $(1,p)$-Poincaré inequalities, see \cite{KeithZhong}, and \Cref{thm: main thm} also imply that the intrinsic parabolic Harnack inequality is stable with respect to the exponent $p$.

\section{Preliminaries}
\label{section: prelim}
\subsection{Geodesic spaces and doubling measures}
The exposition of the preliminary materials in this article largely follows \cite{EleniusKinnunen}. We consider a metric measure space $(X,d,\mu)$, where $X$ is a complete geodesic metric space with metric $d$, and $\mu$ is a Borel measure on $X$. A metric space $X$ is geodesic, if every pair of points in $X$ can be connected by a path with length equal to the distance between the points. We assume that 
the measure $\mu$ is nontrivial in the sense that $0<\mu(B(x_0,r))<\infty$
for every open ball $B(x_0,r)=\{x\in X:d(x,x_0)<r\}$, where $x_0\in X$ and $0<r<\infty$. The above assumptions will be in force throughout the rest of the article. Any additional assumptions  will be explicitly stated when they are imposed.

The measure $\mu$ is doubling if there exists a constant $c_\mu\geq 1$ such that 
\begin{equation}
    \label{eq: doubling condition}
    \mu(B(x_0,2r))\leq c_\mu \mu(B(x_0,r))
\end{equation}
for every $x_0\in X$ and $r>0$.
If $\mu$ is doubling, then for $Q=\log_2c_\mu$ we have
\begin{equation}
    \label{eq: doubling implies radii control from below}
    \frac{\mu(B(x_0,r))}{\mu(B(x_0,R))}\geq \frac{1}{c_\mu} \left(\frac{r}{R}\right)^Q
\end{equation}
for every $x_0\in X$ and $0<r<R<\infty$. If $X$ is connected and $\mu$ is doubling, then there exist constants $c>0$ and $\alpha>0$, depending only on $c_\mu$, such that
\begin{equation}
    \label{eq: doubling implies radii control from above}
    \frac{\mu(B(x_0,r))}{\mu(B(x_0,R))}\leq  c \left(\frac{r}{R}\right)^\alpha
\end{equation}
for every $x_0\in X$ and $0<r<R<\frac{1}{3}\diam X$. In particular, this holds in a geodesic space.
In a geodesic space, the doubling condition improves to the following annular decay property.
If $X$ is geodesic and $\mu$ is doubling, then there exist constants $c>0$ and $\beta>0$, depending only on $c_\mu$, such that 
\begin{equation}\label{eq: geodesic doubling becomes annular decay}
    \mu(B(x_0,r)\backslash B(x_0,(1-\delta)r))\leq c \delta^\beta \mu(B(x_0,r))
\end{equation}
for every $x_0\in X$, $r>0$ and $0<\delta<1$.

\subsection{Upper gradients and Sobolev spaces}
A nonnegative Borel function $g$ on $X$ is called an upper gradient
of an extended real valued function $u$ on $X$ if for all paths $\gamma$ joining points $x$ and $y$ in $\Omega$ we have
\begin{equation} \label{ug-cond}
|u(x)-u(y)|\le \int_\gamma g\,ds,
\end{equation}
whenever both $u(x)$ and $u(y)$ are finite, and $\int_\gamma g\,ds=\infty$ otherwise. 
We say that $g$ is a $p$-weak upper gradient of $u$
if the collection $\Gamma$ of rectifiable curves for which \eqref{ug-cond} fails has zero $p$-modulus, that is, there exists a nonnegative Borel function $f\in L^p(X;\mu)$
such that the integral $\int_\gamma f\, ds$ is infinite for every $\gamma\in \Gamma$. 

Let $\Omega$ be an open subset of $X$ and $1\leq p<\infty$. Let 
\begin{equation}
    \label{eq: newtonian space norm def}
    \Vert u\Vert_{N^{1,p}(\Omega)}=\Vert u\Vert_{L^{p}(\Omega)}+\inf\Vert g\Vert_{L^{p}(\Omega)},
\end{equation}
where the infimum is taken over all upper gradients $g$ of $u$ and 
define $\widetilde{N}^{1,p}(\Omega)$ as the space of functions $u$ for which $\Vert u\Vert_{N^{1,p}(\Omega)}<\infty$.
The Newton--Sobolev space is defined as the quotient space
\[
N^{1,p}(\Omega)=\lbrace u:\Vert u\Vert_{N^{1,p}(\Omega)}<\infty\rbrace/\sim,
\]
where $u\sim v$ if and only if $\Vert u-v\Vert_{N^{1,p}(\Omega)}=0$. We define $N^{1,p}_0(\Omega)$ as the space of functions $u\in N^{1,p}(X)$ that vanish $p$-quasieverywhere in $X\setminus\Omega$, see \cite[Definition 1.24, Definition 2.34]{BjornBjorn2011}. 

If $u$ has an upper gradient $g\in L^p(\Omega)$, there exists a unique minimal $p$-weak upper gradient $g_u\in L^p(\Omega)$ satisfying
$g_u\le g$ $\mu$-almost everywhere for all $p$-weak upper gradients $g\in L^p(\Omega)$ of $u$, see \cite[Theorem 6.3.20]{HeinonenKoskelaShanmugalingam}. Moreover, the minimal $p$-weak upper gradient is unique up to sets of measure zero.
For $u\in N^{1,p}(\Omega)$ we have
\[
\Vert u\Vert_{N^{1,p}(\Omega)}=\Vert u\Vert_{L^{p}(\Omega)}+\Vert g_u\Vert_{L^{p}(\Omega)},
\]
where $g_u$ is the minimal $p$-weak upper gradient of $u$.

We collect some calculus rules for upper gradients on metric measure spaces. 
Let $u,v\in N^{1,p}(\Omega)$ and let $g_{u},\,g_{v}\in L^{p}(\Omega)$ be their minimal $p$-weak upper gradients. 
Then $g_{u}+g_{v}$ and $\vert u\vert g_{v}+\vert v\vert g_{u}$ are $p$-weak upper gradients for $u+v$ and $uv$, respectively, see \cite[Theorem 2.15]{BjornBjorn2011}.
Let  $\eta$ be Lipschitz function on $\Omega$ with $0\leq\eta\leq 1$ and consider $w=u+\eta(u-v)=(1-\eta)u+\eta v$. Then
$(1-\eta)g_{u}+\eta g_{v}+\vert v-u\vert g_{\eta}$
is a $p$-weak upper gradient of $w$, see \cite[Theorem 2.18]{BjornBjorn2011}.
Moreover, $g_{u}=g_{v}$, $\mu$-almost everywhere on the set $\lbrace x\in X: u(x)=v(x)\rbrace$.
In particular, if $c\in\mathbb{R}$ is a constant, then $g_{u}=0$ $\mu-$almost everywhere on the set $\lbrace x\in X: u(x)=c\rbrace$, see \cite[Corollary 2.21]{BjornBjorn2011}.

\subsection{Poincaré inequalities}
The integral average over a $\mu$-measurable subset $A$ of $X$, with $0<\mu(A)<\infty$, is denoted by
\[
u_{A}=\dashint_{A} u\dmu
=\frac{1}{\mu(A)}\int_{A}u\dmu.
\]

\begin{definition}\label{def: poincare}
Let $1\le p,q<\infty$. We say that a  $(q,p)$-Poincaré inequality holds in $(X,d,\mu)$ if there exist constants $c_P>0$ and $\tau\geq 1$ such that 
\begin{equation*}
    \biggl(\dashint_{B(x_0,r)}\lvert u-u_{B(x_0,r)}\rvert^q\dmu\biggr)^{\frac{1}{q}}\leq c_P r \biggl(\dashint_{B(x_0,\tau r)}  g_u^p\dmu\biggr)^\frac{1}{p}
\end{equation*}
for every $u\in N^{1,p}(B(x_0,\tau r))$, where $x_0\in X$ and $0<r<\infty$.
\end{definition}
We note that, in a geodesic space, a  $(q,p)$-Poincaré inequality with $\tau>1$ implies a  $(q,p)$-Poincaré inequality with $\tau=1$, see \cite[Theorem 4.34]{BjornBjorn2011}. When applying the Poincaré inequality in  \Cref{section: VD + PI implies PHI} we will apply it with $\tau=1$. However, when proving a $(1,p)$-Poincaré inequality in \Cref{section: PHI implies Poincare} we will prove it for a parameter $\tau>1$.

If $\mu$ is doubling and a $(1,p$)-Poincaré inequality holds in $(X,d,\mu)$ with $p<Q$, then a $(p^*,p)$-Poincaré inequality holds, where $p^*=\frac{pQ}{Q-p}$. Moreover, if $p\geq Q$ then a $(p^*,p)$-Poincaré inequality holds for every $1\leq p^*<\infty$, 
see \cite[Theorem 9.1.15]{HeinonenKoskelaShanmugalingam} or \cite[Theorem 4.16]{BjornBjorn2011}. Moreover, there exists $q<p$ depending on $c_\mu$, $C_p$ and $p$, such that a $(1,q)$-Poincaré inequality holds in $(X,d,\mu)$, see \cite{KeithZhong}. 

For the following result for functions that vanish on a large portion of a set we refer to \cite[Lemma 2.1]{KinnunenShanmugalingam}.
\begin{lemma}
\label{eq: Poincaré for function zero on large set}
Let $1\le p<\infty$ and $1<q<\infty$. Assume that $\mu$ is doubling and that a  $(q,p)$-Poincaré inequality holds in $(X,d,\mu)$. Let $u\in N^{1,p}(X)$ and assume that there exists $0<\gamma<1$ such that $\mu(\{ \lvert u\rvert >0\}\cap B(x_0,r)) \leq \gamma \mu(B(x_0,r))$, where $x_0\in X$ and $0<r<\infty$. Then there exists a constant $c$, depending on $c_\mu$, $c_P$, $p$, $q$ and $\gamma$, such that
\begin{equation*}
    \biggl( \dashint_{B(x_0,r)} \lvert u\rvert^{q} \dmu \biggr)
^\frac{1}{q}\leq c r \biggl( \dashint_{B(x_0, \tau r)} g_u^p\dmu \biggr)^\frac{1}{p}.
\end{equation*}
\end{lemma}
The previous lemma implies the following Sobolev inequality for functions with zero boundary values, see \cite{KinnunenShanmugalingam}.
\begin{lemma}
\label{lemma: sobolev from poincaré}
Let $1\le p<\infty$ and $1<q<\infty$. Assume that $\mu$ is doubling and that a $(q,p)$-Poincaré inequality holds in $(X,d,\mu)$. Then there exists a constant $c$, depending on $c_\mu$, $c_P$, $p$ and $q$, such that
\begin{equation*}
    \biggl(\dashint_{B(x_0,r)} \lvert u\rvert^q \dmu\biggr)^\frac{1}{q} \leq cr \biggl(\dashint_{B(x_0,r)} g_u^p \dmu\biggr)^\frac{1}{p}
\end{equation*}
for every $u\in N^{1,p}_0(B(x_0,r))$, where $x_0\in X$ and $0<r\leq\frac13\diam(X)$.
\end{lemma}
\subsection{Parabolic Sobolev spaces}
We briefly discuss time dependent Sobolev spaces that are relevant to the doubly nonlinear equation \eqref{eq: PDE}.
\begin{definition}\label{def.pnewtonian}
Let $\Omega$ be an open subset of $X$, $0<T<\infty$ and $1\le p<\infty$. 
The parabolic Newton--Sobolev space $L^{p}(0, T;N^{1,p}(\Omega))$ consists of  strongly measurable functions 
$u:(0,T)\to N^{1,p}(\Omega)$ with the norm
\begin{equation*}
\Vert u\Vert_{L^{p}(0, T;N^{1,p}(\Omega))}
=\biggl(\int_{0}^{T}\Vert u(t)\Vert_{N^{1,p}(\Omega)}^{p}\dt\biggr)^{\frac1p}<\infty.
\end{equation*}
The integration over time is taken with respect to the one-dimensional Lebesgue measure.

It is also possible to consider parabolic Newton--Sobolev spaces over other time intervals, for example, $L^{p}(t_1,t_2;N^{1,p}(\Omega))$, $-\infty<t_1<t_2<\infty$, or $L^{p}(-\infty,\infty;N^{1,p}(\Omega))$.
\end{definition}

The product measure in the space $X\times (0,T)$, $0<T<\infty$, is denoted by $\nu$.
If $u\in L^{p}(0, T;N^{1,p}(\Omega))$, it is immediate that $u\in L^{p}(0, T;L^{p}(\Omega))$. The space-time cylinder is denoted by $\Omega_{T}=\Omega\times(0, T)$. For $u\in L^{p}(0, T;L^{p}(\Omega))$, there exists a $\nu$-measurable representative $u:\Omega_T\to[-\infty,\infty]$ such that $u(t)=u(\cdot,t)$ for almost every $0<t<T$ and
\[
\int_{\Omega_T}|u|^p\dnu
=\int_0^T\int_\Omega|u(x,t)|^p\dmu\dt
=\int_0^T\Vert u(t)\Vert_{L^p(\Omega)}^p\dt.
\]
If $u\in L^{p}(0, T;N^{1,p}(\Omega))$, we may conclude that the map $t\mapsto g_{u(t)}$ belongs to $L^{p}(0, T;L^{p}(\Omega))$. Moreover, we find a $\nu$-measurable representative
$g_u:\Omega_T\to[-\infty,\infty]$ such that $g_u(t)=g_u(\cdot,t)$ for almost every $0<t<T$ and 
\[
\int_{\Omega_T}g^p\dnu
=\int_0^T\int_\Omega g_u(x,t)^p\dmu\dt
=\int_0^T\Vert g_u(t)\Vert_{L^p(\Omega)}^p\dt.
\]
We shall implicitly identify $u$ and $g_u$ with their $\nu$-measurable representatives.

The parabolic Newton--Sobolev space with zero lateral boundary values is defined by saying that $u\in L^p(0,T; N^{1,p}_0(\Omega))$ if $u\in L^{p}(0, T;N^{1,p}(\Omega))$ and $u(t)\in N^{1,p}_0(\Omega)$ for almost every $0<t<T$. The following is a parabolic Sobolev embedding for functions $u\in L^p(0,T; N^{1,p}_0(\Omega))\cap L^\infty(0,T;L^{q+1}(X))$.
\begin{lemma}
\label{lemma: par sob}
Let $1<p<\infty$ and $0<q<\infty$. Assume that $\mu$ is doubling and that a $(1,p)$-Poincaré inequality holds in $(X,d,\mu)$. Then there exists $\kappa>0$, depending only on $p$, and $c_\mu$, and a constant $c>0$ depending on $c_\mu$, $c_P$ and $p$ such that for every ball $B(x_0,r)$ with $r\leq\diam(X)/3$ we have
\begin{equation*}
    \int_{0}^{T}\dashint_{B(x_0,r)} \lvert u\rvert^{p+\kappa(q+1)}\dmu\dt \leq c r^p \int_{0}^{T}\dashint_{B(x_0,r)} g_u^p\dmu\dt \esssup_{0<t<T}\left( \dashint_{B(x_0,r)} u^{q+1}\dmu \right)^{\kappa},
\end{equation*}
for every $u\in L^p(0,T;N^{1,p}_0(B(x_0,r))\cap L^\infty(0,T;L^{q+1}(B(x_0,r))$.
\end{lemma}
\begin{proof}
Let $\kappa=1-\frac{p}{p^*}>0$ and let $t\in(0,T)$. Hölder's inequality and \Cref{lemma: sobolev from poincaré} implies
\begin{equation*}
\begin{split}
    \dashint_{B(x_0,r)} \lvert u\rvert^{p+\kappa(q+1)}\dmu & \leq \left( \dashint_{B(x_0,r)} \lvert u\rvert^{p^*} \dmu\right)^\frac{p}{p^*} \left(\dashint_{B(x_0,r)} \lvert u\rvert^{q+1}\dmu\right)^{\kappa} \\
    & \leq cr^p \dashint_{B(x_0,r)} g_u^p \dmu \left(\dashint_{B(x_0,r)} \lvert u\rvert^{q+1}\dmu\right)^{\kappa}.
\end{split}
\end{equation*}
Integrating the above over $t\in (t,T)$ shows the claim.
\end{proof}

\section{Doubly nonlinear equations}
In this section we give the definition of variational solutions to \eqref{eq: PDE} and obtain energy estimates. We will use the following convention for powers of signed numbers.
\begin{equation}\label{eq: signconv}
    b^\alpha = \begin{cases}
            \lvert b\rvert^{\alpha-1}b ,&\quad b \neq 0,\\
            0, & \quad b=0. 
        \end{cases}
\end{equation}

\begin{definition}\label{def.soldne}
Let $1<p<\infty$, $0<q<\infty$, $\Omega$ be an open subset of $X$ and $0<T<\infty$. A function $u\in L^p(0,T;N^{1,p}(\Omega))\cap L^{q+1}(\Omega_T)$ is a  solution of \eqref{eq: PDE}
in $\Omega_T=\Omega\times(0,T)$ if 
\begin{equation}
\label{eq: var ineq}
    p\int_{\Omega_T}u^{q}\partial_t{\varphi}\dnu  + \int_{\Omega_T} g_u^p\dnu \leq \int_{\Omega_T} g_{u+\varphi}^p\dnu,
\end{equation}
for every $\varphi\in L^p(0,T;N^{1,p}(\Omega))\cap L^{q+1}(X)$ with $\partial_t \varphi\in L^{q+1}(\Omega_T)$ and  $\supp \varphi\subset\Omega_T$. If \eqref{eq: var ineq} holds for all such $\varphi$ with $\varphi\geq 0$ we say that $u$ is a  supersolution. 
If \eqref{eq: var ineq} holds for all such $\varphi$ with $\varphi\leq 0$ we say that $u$ is a subsolution.
\end{definition}

The following direct consequence of the definition of a solution will be useful.
\begin{lemma}
\label{lemma: var sol implies weak sol with inequality}
Let $1<p<\infty$ and $0<q<\infty$. Assume that $u$ is a solution of \eqref{eq: PDE} in $\Omega_T$. Then
\begin{equation*}
    \int_{\Omega_T} u^{q}\partial_t\varphi \dnu \leq \int_{\Omega_T} g_u^{p-1}g_\varphi\dnu 
\end{equation*}
for every $\varphi\in L^p(0,T;N^{1,p}(\Omega))$ with $\partial_t \varphi\in L^{q+1}(\Omega_T)$ and  $\supp \varphi\subset\Omega_T$.
\end{lemma}
\begin{proof}
Let $h>0$ and test \eqref{eq: var ineq} with $h\varphi$ to get
\begin{equation*}
    p\int_{\Omega_T} u^{q} \partial_t{\varphi}\dnu  \leq \frac{1}{h}\int_{\Omega_T}(g_{u+h\varphi}^p-g_u^p)\dnu
\end{equation*}
The mean value theorem implies that
\begin{equation*}
    g_{u+h\varphi}^p-g_u^p\leq (g_u +h g_{\varphi})^p-g_u^p \leq p (g_u + hg_\varphi)^{q}hg_\varphi,
\end{equation*}
and hence 
\begin{equation*}
    p\int_{\Omega_T}u^{q}\partial_t{\varphi}\dnu  \leq p\int_{\Omega_T} (g_u + hg_\varphi)^{q}g_\varphi\dnu.
\end{equation*}
The claim follows by letting $h\to 0$.
\end{proof}

\subsection{Time mollification}
We use the following mollification in time when proving energy estimates.
Let $u\in L^p(0,T;N^{1,p}(\Omega))$. For $\delta>0$, we define 
\begin{equation*}
    u_{\delta}(t) =\frac{1}{\delta}\int_0^te^{\frac{s-t}{\delta}}u(x,s)\ds
\end{equation*}
For the following lemma we refer to \cite[Lemma B.1]{BogeleinDuzaarMarcellini}.
\begin{lemma}
Let $1\leq p<\infty$ and $u\in L^p(0,T;N^{1,p}(\Omega))$.
Then $u_{\delta}\in L^p(0,T;N^{1,p}(\Omega))$ and 
\begin{equation*}
    \|u_\delta\|_{L^p(0,T;N^{1,p}(\Omega))} \leq \|u\|_{L^p(0,T;N^{1,p}(\Omega))}
\end{equation*}
for every $\delta>0$.
Moreover, we have $\partial_t u_\delta\in L^p(0,T;N^{1,p}(\Omega))$ and 
\begin{equation}
\label{eq: time der of mollification}
    \partial_t u_\delta = \frac{1}{\delta}(u-u_\delta)
\end{equation}
for every $\delta>0$.
\end{lemma}
We also have the convergence in norm and convergence of mixed terms, as in the following lemma, see \cite[Lemma 3.4]{EleniusKinnunen}
\begin{lemma}
\label{lemma: convergence of time mollification}
Let $1\leq p<\infty$ and $u\in L^p(0,T;N^{1,p}(\Omega))$. Then 
\begin{equation*}
    \lim_{\delta\to 0} \| u_\delta -u\|_{L^p(0,T;N^{1,p}(\Omega))}=0,
\end{equation*}
and
\begin{equation*}
    \lim_{\delta \to 0} \frac{1}{\delta}\int_0^T \int_0^{T-t}\| u(t)-u(t+s)\|_{N^{1,p}(\Omega)}^p e^{-\frac{s}{\delta}}\ds\dt=0.
\end{equation*}
\end{lemma}
Though we defined $u_{\delta}$ only for $u\in L^p(0,T;N^{1,p}(\Omega))$, the definition makes sense more generally for $u\in L^p(0,T;L^{p}(\Omega))$. Convergence of $u^\delta$ to $u$ as $\delta\to 0$ holds in the norm of $L^p(0,T;N^{1,p}(\Omega))$ as well.
\subsection{Energy estimate}
Given $0<q<\infty$, let
\begin{equation*}
    G_{\pm}(w,k)=\pm q \int_{k}^{w}\lvert s\rvert^{q-1}(s-k)_{\pm}\ds.
\end{equation*}
A change of variables shows that 
\begin{equation}
    \label{eq: primitive for energy est after change of var}
    G_{\pm}(w,k)=\int_{ k^{q}}^{ w^{q}}(s^\frac{1}{q}-k)_\pm\ds.
\end{equation}
The following lemma is from \cite[Lemma 2.2]{BogeleinDuzaarLiao}.
\begin{lemma}
Let $0<q<\infty$. There exists a constant $c=c(q)$ such that
\begin{equation*}
    \frac{1}{c}(\lvert w\rvert+\lvert k\rvert)^{q-1}(w-k)_\pm^2 \leq G_{\pm}(w,k)\leq c (\lvert w\rvert+\lvert k\rvert)^{q-1}(w-k)_\pm^2,
\end{equation*}
for all $w,k\in \mathbb R$.
\end{lemma}
The following lemma is an energy estimate for subsolutions to \eqref{eq: PDE}.
\begin{lemma}
\label{lemma: caccioppoli with cutoff}
Let $1<p<\infty$ and $0<q<\infty$. Let $u$ be a local subsolution in $\Omega_T$. There exists a constant $c$, depending on $p$, such that for any $k\in\mathbb R$ and a.e. $0<t_0<T$, there holds 
\begin{equation*}
\begin{split}
    & \esssup_{t_0<t<T}p\int_{\Omega} G_{+}(u(t),k)\varphi^p\dmu + \int_{\Omega_{T}}g_{(u-k)_+}^p\varphi^{p}\dmu\dt \\
    & \qquad \leq c\int_{\Omega_{T}}(u-k)_+^p g_{\varphi}^p+G_+(u,k)(\partial_t\varphi^p)_+\dmu\dt + p\int_{\Omega} G_+ (u(t_0),k)\varphi^{p}\dmu .
\end{split}
\end{equation*}
for every non-negative Lipschitz $\varphi\in N^{1,p}(0,T;N^{1,p}(\Omega))$ with $\supp \varphi \subset \Omega_T$.
\end{lemma}
\begin{proof}
Let $\varphi$ be as in the statement. We may assume that $0\leq \varphi\leq 1$. Since the support of $\varphi$ is compact, after a change of variables we may assume that $0$ is a Lebesgue point of $u$ in the sense that 
\begin{equation*}
    \lim_{h\to 0}\frac{1}{h}\int_{0}^{h} \int_{\Omega} \lvert u - u(\cdot,0) \rvert^{q+1}\dmu\dt=0.
\end{equation*}
Let $\delta>0$. The mollified form of \eqref{eq: var ineq} is 
\begin{equation}
\label{eq: mollified var ineq}
\begin{split}
    -p\int_{\Omega_T} \partial_t(u^q)_\delta \eta\dmu\dt  +& \int_{\Omega_T} (g_u^p)_\delta \dmu\dt \leq \frac{1}{\delta}\int_0^T\int_0^{T-t}\int_{\Omega} g_{u(t)+\eta(t+s)}^p\dmu\ds\dt\\
    &+\int_{\Omega} u^{q}(\cdot,0)\frac{1}{\delta}\int_0^T e^{-\frac{s}{\delta}}\eta(x,s)\ds\dmu   ,
\end{split}
\end{equation}
for any $\eta\in L^p(0,T;N_0^{1,p}(\Omega))\cap L^{q+1}(\Omega_T)$. Let $0<\tau_1<\tau_2<T$ and $\epsilon>0$. Construct a cutoff function
\begin{equation*}
    \xi(t)
        =\begin{cases}
            0 ,&\quad 0<t\leq \tau_1-\epsilon\\
            \frac{1}{\epsilon}(t-\tau_1+\epsilon), & \quad \tau_1-\epsilon <t\leq \tau_1 \\
            1,&\quad \tau_1<t<\tau_2,\\
            \frac{1}{\epsilon}(\tau_2+\epsilon-t), & \quad \tau_2<t\leq \tau_2+\epsilon\\
            0,&\quad \tau_2 +\epsilon <t<T .
        \end{cases}
\end{equation*}
We apply $\eta=-\xi\varphi^p  (u-k)_+$ as a test function in \eqref{eq: mollified var ineq}. For the term containing the time derivative, we see that
\begin{equation*}
\begin{split}
    -p\int_{\Omega_T} \partial_t(u^q)_\delta \eta\dmu\dt & = p\int_{\Omega_T} \partial_t (u^q)_\delta \xi\varphi^p(u-k)_+\dmu\dt \\
    & = p\int_{\Omega_T} \partial_t(u ^{q})_\delta \xi\varphi^p( ( u^q)_\delta^{\frac{1}{q}}-k)_+\dmu\dt \\
    & \qquad + p\int_{\Omega_T} \partial_t(u^{q})_\delta \xi\varphi^p( (u-k)_+ -( u^{q})_\delta^{\frac{1}{q}}-k)_+\dmu\dt.
\end{split}
\end{equation*}
Using \eqref{eq: time der of mollification} and the fact that the map $\lambda\mapsto ( \lambda^\frac{1}{q}-k)_+$ is a monotone increasing function, we obtain
\begin{equation*}
\begin{split}
    &p\int_{\Omega_T} \partial_t(u^{q})_\delta \xi\varphi^p( (u-k)_+ -(u^{q})_\delta^{\frac{1}{q}}-k)_+\dmu\dt \\
    &= p\int_{\Omega_T} \frac{1}{\delta}(u^{q}-( u^{q})_\delta) \xi\varphi^p( (u-k)_+ -( u^{q})_\delta^{\frac{1}{q}}-k)_+\dmu\dt \\
    & \geq 0.
\end{split}
\end{equation*}
This implies that
\begin{equation*}
\begin{split}
    -p\int_{\Omega_T} \partial_t( u^q)_\delta \eta\dmu\dt &\geq 
     p\int_{\Omega_T} \partial_t(u ^{q})_\delta \xi\varphi^p( ( u^{q})_\delta^{\frac{1}{q}}-k)_+\dmu\dt \\
     & = - p\int_{\Omega_T} \partial_t(\varphi^p \xi) G_+(( u^{q})_\delta^\frac{1}{q},k)\dmu\dt,
\end{split}
\end{equation*}
where in the second line we used \eqref{eq: primitive for energy est after change of var} to conclude that 
\begin{equation*}
    \partial_t G_+((u^{q})_\delta^\frac{1}{q},k)= \partial_t ( u^{q})_\delta (( u^{q})_\delta^\frac{1}{q}-k)_+.
\end{equation*}
Using \eqref{eq: primitive for energy est after change of var} and the convergence $( u^{q})_\delta\to   u^{q}$ in $L^{\frac{q+1}{q}}$ as $\delta\to 0$, we also deduce that
\begin{equation*}
    \lim_{\delta\to 0}\int_{\Omega_T} \partial_t(\varphi^p \xi) G_+((u^{q})_\delta^\frac{1}{q},k)\dmu\dt = \int_{\Omega_T} \partial_t(\varphi^p \xi) G_+(u,k)\dmu\dt.
\end{equation*}
Applying \Cref{lemma: convergence of time mollification} and the fact that $\eta$ has compact support in $(0,T)$ yields
\begin{equation*}
\begin{split}
    & \lim_{\delta\to 0} \frac{1}{\delta}\int_0^T\int_0^{T-t}\int_{\Omega} g_{u(t)-\eta(t+s)}^p\dmu\ds\dt+\int_{\Omega}  u^{p-1}(0)\frac{1}{\delta}\int_0^T e^{-\frac{s}{\delta}}\eta(x,s)\ds\dmu \\
    &\qquad = \int_{\Omega_T} g_{u-\eta}^p\dmu\dt.
\end{split}
\end{equation*}
Hence, letting $\delta\to 0$ in \eqref{eq: mollified var ineq} we obtain
\begin{equation*}
\begin{split}
    & -p\int_{\Omega_T} \partial_t(\varphi^p \xi)G_+(u,k)\dmu\dt + \int_{\Omega_T} g_{u}^p\dmu\dt  \leq \int_{\Omega_T} g_{u-\eta}^p\dmu\dt.
\end{split}
\end{equation*}
The product rule yields 
\begin{equation*}
\begin{split}
     &-p\int_{\Omega_T} \partial_t(\varphi^p \xi)G_+(u,k)\dmu\dt + \int_{\Omega_T} g_{(u-k)_+}^p\dmu\dt \leq  \int_{\Omega_T} g_{(1-\xi\varphi^p)(u-k)_+}^p\dmu\dt \\
    &\qquad \leq \int_{\Omega_T} ((1-\xi\varphi^p)g_{(u-k)_+}+\xi p\varphi^{p-1}g_\varphi (u-k)_+ )^p\dmu\dt.
\end{split}
\end{equation*}
The convexity of the mapping $t\mapsto t^p$ implies that
\begin{equation*}
\begin{split}
    ((1-\xi\varphi^p)g_{(u-k)_+}+\xi p\varphi^{p-1}g_\varphi (u-k)_+ )^p & = ((1-\xi\varphi^p)g_{(u-k)_+}+\xi\varphi^{p}\frac{p}{\varphi}g_\varphi (u-k)_+ )^p \\
    & \leq (1-\xi \varphi^p)g_{(u-k)_+}^p+\xi p^p g_\varphi^p(u-k)_+,
\end{split}
\end{equation*}
and hence it holds that
\begin{equation*}
\begin{split}
    &-p\int_{\Omega_T}  (\partial_t \varphi^p \xi + \partial_t \xi \varphi^p)G_+(u,k)\dmu\dt + \int_{\Omega_T} g_{(u-k)_+}^p\dmu\dt \\
    &\qquad\leq \int_{\Omega_T} (1-\xi\varphi^p)g_{(u-k)_+}^p\dmu\dt + p^p \int_{\Omega_T} \xi g_{\varphi}^p (u-k)_+^p\dmu\dt.
\end{split}
\end{equation*}
Absorbing the first term on the right hand side of the above, we have
\begin{equation*}
\label{eq: energy estimate before time der calc}
\begin{split}
&-p\int_{\Omega_T}  (\partial_t \varphi^p \xi + \partial_t \xi \varphi^p)G_+(u,k)\dmu\dt + \int_{\Omega_T} g_{(u-k)_+}^p\xi\varphi^p\dmu\dt \\
& \qquad \leq p^p \int_{\Omega_T} g_{\varphi}^p (u-k)_+^p\dmu\dt.    
\end{split}
\end{equation*}
Letting $\epsilon \to 0$ in the above yields
\begin{equation}
\label{eq: energy estimate before choice of times}
\begin{split}
    &p\int_{\Omega}  G_+(u(\tau_2),k)\dmu + \int_{\Omega_{\tau_2,\tau_2}} g_{(u-k)_+}^p\varphi^p\dmu\dt \leq p^p \int_{\Omega_T} g_{\varphi}^p (u-k)_+^p\dmu\dt \\
    & \qquad + p\int_{\Omega_T} G_+(u,k)\partial_t(\varphi^p)\dmu\dt+ p\int_{\Omega} G_+(u(\tau_1),k)\varphi^p\dmu,
\end{split}
\end{equation}
for almost every $\tau_1,\tau_2\in (0,T)$. First, letting $\tau_1\to 0$ and $\tau_2\to T$ in \eqref{eq: energy estimate before choice of times}, we obtain
\begin{equation}
\label{eq: energy estimate without esssup}
    \int_{\Omega_{T}} g_{(u-k)_+}^p\varphi^p\dmu\dt\leq p^p \int_{\Omega_T} g_{\varphi}^p (u-k)_+^p\dmu\dt + p\int_{\Omega_T} G_+(u,k)(\partial_t\varphi^p)_+\dmu\dt.
\end{equation}
Second, discarding the second term on the left hand side of \eqref{eq: energy estimate before choice of times} and letting $\tau_1=t_0$, $\tau_2\in(t_0,T)$ gives
\begin{equation}
\label{eq: esssup term for energy est}
\begin{split}
    & \esssup_{t_0<t<T} p\int_{\Omega} G_+(u(t),k)\dmu \leq p^p \int_{\Omega_T} g_{\varphi}^p (u-k)_+^p\dmu\dt \\
    & \qquad + p\int_{\Omega_T} G_+(u,k)\partial_t(\varphi^p)\dmu\dt+ p\int_{\Omega} G_+(u(t_0),k)\varphi^p\dmu.
\end{split}
\end{equation}
Adding \eqref{eq: energy estimate without esssup} and \eqref{eq: esssup term for energy est} we arrive at 
\begin{equation*}
\begin{split}
    & \esssup_{t_0<t<T}p\int_{\Omega} G_{+}(u(t),k)\varphi^p\dmu + \int_{\Omega_{T}}g_{(u-k)_+}^p\varphi^p\dmu\dt \\
    &\qquad  \leq c\int_{\Omega_{T}}(u-k)_+^p g_{\varphi}^p+G_+(u,k)(\partial_t\varphi^p)_+\dmu\dt + p\int_{\Omega} G_+ (u(t_0),k)\varphi^p\dmu,
\end{split}
\end{equation*}
which finishes the proof.
\end{proof}

\section{Poincaré and doubling implies parabolic Harnack inequality}
\label{section: VD + PI implies PHI}
In this section we will establish that the conjunction of volume doubling and a $(1,p)$-Poincaré inequality imply the intrinsic parabolic Harnack inequality. For $r>0$, $x_0\in X$, $t_0\in \R$ and $\delta>0$, we will use the notation 
\begin{equation*}
    Q^-_{r,\delta}(x_0,t_0)=B(x_0,r)\times(t_0-\delta r^p,t_0),
\end{equation*}
and $Q^-_{r}(x_0,t_0)=Q^-_{r,1}(x_0,t_0)$.
\subsection{Local boundedness}
We begin by showing the following proposition, which shows that subsolutions are locally bounded from above.
\begin{proposition}
\label{prop: esssup estimate}
Let $1<p<\infty$ and $0<q<\infty$ with $q+1<p$. Assume that $\mu$ is doubling and that a $(1,p)$-Poincaré inequality holds in $(X,d,\mu)$. Let $u$ be a subsolution in $\Omega_T$. There exists a constant $c$, depending on $c_\mu$, $c_P$, $p$ and $q$ such that if $T>0$, $0<r<\diam(X)/3$ and $(x_0,t_0)\in \Omega_T$ with $B(x_0,r)\times (t_0-T,t_0)\subset\Omega_T$ then
\begin{equation*}
    \esssup_{B(x_0,r/2)\times(t_0-T/2,t_0)} u\leq c \left(\frac{T}{r^p} \dashint_{B(x_0,r)\times(t_0-T,t_0)} u_+^{p}\dnu\right)^\frac{1}{q+1} + \left( \frac{r^p}{T}\right)^\frac{1}{p-q-1}
\end{equation*}
\end{proposition}

\begin{proof}
For $n\in \mathbb N$ and $k>0$, let
\begin{equation*}
k_n = (1-2^{-n})k, \quad \Tilde k_n=\frac{k_{n+1}+k_n}{2}, 
\quad r_n = \frac{r}{2} + \frac{r}{2^{n+1}}, \quad T_n = \frac{T}{2}+\frac{T}{2^{n+1}},
\end{equation*}
$B_n =B(x_0,r_n)$ and $Q_n= B(x_0,r)\times(t_0-T_n,t_0)$. At this stage, we choose $k$ such that 
\begin{equation}
    \label{eq: loc bdd first k choice}
    k \geq (r^p/T)^\frac{1}{p-q-1}
\end{equation}
Let $\eta_n$ be a $\frac{c}{r_{n}-r_{n+1}}$-Lipschitz cutoff function with $0\leq \eta\leq 1$, $\eta=1$ on $B_{n+1}$ and $\supp \eta \subset B_n$. Denote
\begin{equation*}
    \xi_n(t)
        =\begin{cases}
            0 ,&\quad t\leq t_0-T_n,\\
            \frac{1}{T_{n}-T_{n+1}}(t-t_0+T_n), & \quad t_0-T_n <t\leq t_0-T_{n+1}, \\
            1,&\quad t\geq t_0-T_{n+1},
        \end{cases}
\end{equation*}
and denote $\varphi_n=\eta_n\xi_n$. Applying \Cref{lemma: par sob} we have
\begin{equation}
\label{eq: loc bdd from par sob}
\begin{split}
    &\dashint_{Q_{n+1}}(u-k_{n+1})_+^p\dnu\\
    &\qquad= (k_{n+1}-\Tilde k_n)^{\kappa(q+1)} \dashint_{Q_{n+1}}(u-k_{n+1})_+^p(k_{n+1}-\Tilde k_n)^{\kappa(q+1)}\dnu \\
    &\qquad\leq ck^{\kappa(q+1)}2^{n\kappa(q+1)}\dashint_{Q_n}((u-\Tilde k_n)_+\varphi_n)^{p+\kappa(q+1)}\dnu \\
    &\qquad\leq ck^{\kappa(q+1)}2^{n\kappa(q+1)} r_n^p \dashint_{Q_n}g_{(u-\Tilde k_n)_+\varphi_n}^p\dnu
     \esssup_{t_0-T_n<t<t_0}\biggl(\dashint_{B_n}(u(t)-\Tilde k_n)_+^{q+1} \varphi_n^p\dmu\biggr)^{\kappa},
\end{split}
\end{equation}
where $\kappa>0$ is as in \Cref{lemma: par sob}. We claim that 
\begin{equation}
\label{eq: wanted est for loc bdd from cacciop}
    \begin{split}
     &\esssup_{t_0-T_n<t<t_0}\dashint_{B_n} (u(t)-\Tilde k_n)_+^{q+1} \varphi_n^p\dmu + T_n \dashint_{Q_n}g_{(u-\Tilde k_n)_+ \varphi_n }^p\dnu\\
     &\qquad\leq c  \frac{T_n}{r_n^p} 2^{(2p+2q+2)n}\dashint_{Q_n}(u-k_n)^p\dnu.
     \end{split}
\end{equation}
Let $h>0$. Applying \Cref{lemma: caccioppoli with cutoff} we have
\begin{equation}
\label{eq: general est for loc bdd from cacciop}
\begin{split}
    & \esssup_{t_0-T_n<t<t_0}\dashint_{B_n} (u(t) +h)^{q-1}(u(t)-h)_+^2 \varphi_n^p\dmu + r_n^p\dashint_{Q_n}g_{(u-h)_+}^p\varphi_n^p\dnu \\
    &\qquad \leq c 2^{pn}\dashint_{Q_n}(u-h)_+^p +\frac{r^p}{T_n}( u_+ +h )^{q-1}(u-h)_+^2\dnu.
\end{split}
\end{equation}
We consider two ranges of $q$ separately. First, assume $0<q<1$. Then
\begin{equation*}
\begin{split}
    (u-\Tilde k_n)_+^{q+1} & \leq (u+k'_n)^{q+1} \chi_{u>\Tilde k_n} 
     =(u+k'_n)^{q-1} ((u-k_n')_+ + k_n')^2\chi_{u>\Tilde k_n'} \\
    &\leq c  4^{n} (u+k_n')^{q-1} ((u-k_n')_+ + (\Tilde k_n-k_n'))^2\chi_{u>\Tilde k_n} \\
    &\leq c  4^{n} (u+k_n')^{q-1} (u-k_n')_+^2
    \leq c 4^n (u-k_n')_+^{q+1},
\end{split}
\end{equation*}
where $k_n'=(\Tilde k_n+k_n)/2$. The product rule and \eqref{eq: general est for loc bdd from cacciop} with $h=k_n'$  imply that
\begin{equation}
\label{eq: loc bdd q small cacciop before time scaling k appl}
\begin{split}
     &\esssup_{t_0-T_n<t<t_0}\dashint_{B_n} (u(t)-\Tilde k_n)_+^p \varphi_n^p\dmu + T_n\dashint_{Q_n}g_{(u-\Tilde k_n)_+\varphi_n}^p\dnu \\
     &\qquad\leq c 2^{2n}\esssup_{t_0-T_n<t<t_0}\dashint_{B_n} (u(t)+k_n')^{p-2}(u(t)-k_n')_+^2 \varphi_n^p\dmu\\
     &\qquad\qquad+ c T_n \dashint_{Q_n}g_{(u-k_n')_+}^p\varphi_n^p\dnu
      + c T_n\dashint_{Q_n} (u-\Tilde k_n)_+^p g_{\varphi_n}^p\dnu \\
     &\qquad \leq c \frac{T_n}{r_n^p} 2^{(p+2)n}\dashint_{Q_n}(u-k_n')_+^p+ \frac{r^p}{T}(u-k_n')_+^{q+1}\dnu.
\end{split}
\end{equation}
Recalling \eqref{eq: loc bdd first k choice}, we obtain
\begin{equation}
\label{eq: loc bdd comp k with time scaling for cacciop}
\begin{split}
    \dashint_{Q_n}\frac{r^p}{T}(u-k_n')_+^{q+1}\dnu & \leq \dashint_{Q_n} k^{p-q-1}(u-k_n')_+^{q+1}\dnu \\
    & \leq c 2^{n(p-q-1)} \dashint_{Q_n} (k_n'-k_n)^{p-q-1}(u-k_n')_+^{q+1}\dnu \\
    & \leq c 2^{n(p-q-1)} \dashint_{Q_n} (u-k_n)_+^p\dnu.
\end{split}
\end{equation}
Applying the above in \eqref{eq: loc bdd q small cacciop before time scaling k appl} shows \eqref{eq: wanted est for loc bdd from cacciop} in the case $0<q<1$. Then assume $1\leq q<\infty$. Then, it holds that
\begin{equation*}
\begin{split}
    (u-\Tilde k_n)_+^{q+1} &\leq (u+\Tilde k_n)^{q-1}(u-\Tilde k_n)_+^2 \leq c 2^{(q-1)n}(u-k_n')_+^{q+1}.
\end{split}
\end{equation*}
Applying \eqref{eq: general est for loc bdd from cacciop} with $h=\Tilde k_n$ implies 
\begin{equation*}
\begin{split}
     &\esssup_{t_0-T_n<t<t_0}\dashint_{B_n} (u(t)-\Tilde k_n)_+^p \varphi_n^p\dmu + T_n\dashint_{Q_n}g_{(u-\Tilde k_n)_+\varphi_n^p}^p\dnu \\
     &\qquad \leq \esssup_{t_0-T_n<t<t_0}\dashint_{B_n} (u(t)+\Tilde k_n)^{p-2}(u(t)-\Tilde k_n)_+^2 \varphi_n^p\dmu \\
     &\qquad\qquad+ cT_n\dashint_{Q_n}g_{(u-\Tilde k_n)_+}^p\varphi_n^p\dnu
     + c T_n\dashint_{Q_n}(u-\Tilde k_n)_+^pg_{\varphi_n}^p\dnu \\
     &\qquad\leq c \frac{T_n}{r_n^p} 2^{(p+q-1)n}\dashint_{Q_n}(u-k_n)_+^p\dnu.
\end{split}
\end{equation*}
Arguing similarly as in \eqref{eq: loc bdd comp k with time scaling for cacciop} we obtain that 
\begin{equation*}
\begin{split}
     &\esssup_{t_0-T_n<t<t_0}\dashint_{B_n} (u(t)-\Tilde k_n)_+^p \varphi_n^p\dmu + T_n\dashint_{Q_n}g_{(u-\Tilde k_n)_+\varphi_n^p}^p\dnu \\
     &\qquad\leq c \frac{T_n}{r_n^p} 2^{2(p+q-1)n}\dashint_{Q_n}(u-k_n)_+^p\dnu,
\end{split}
\end{equation*}
and hence \eqref{eq: wanted est for loc bdd from cacciop} holds in this case as well. 

We apply \eqref{eq: wanted est for loc bdd from cacciop} in \eqref{eq: loc bdd from par sob} to conclude that
\begin{equation*}
    \frac{T}{r^p }\dashint_{Q_{n+1}}(u-k_{n+1})_+^p\dnu \leq ck^{\kappa(q+1)}2^{\kappa(2p+3q+3)n}\biggl(\frac{T}{r^p}\dashint_{Q_n} (u-k_n)_+^p\dnu \biggr)^{1+\kappa}.
\end{equation*}
Denoting
\begin{equation*}
    Y_n= \frac{T}{r^p}\dashint_{Q_n}(u-k_n)_+^p\dnu 
    \quad\text{and}\quad 
    b=2^{\kappa(2p+3q+3)n},
\end{equation*}
we have shown
\begin{equation*}
    Y_{n+1}\leq ck^{\kappa(q+1)}b^n Y_n^{1+\kappa}
\end{equation*}
for every $n\in\mathbb N$.
Applying \cite[Lemma 5.1]{DiBenedettoGIanazzaVespri2012} we obtain $Y_n\to 0$ as $n\to \infty$ provided that 
\begin{equation*}
    Y_0 \leq c^{-1/\kappa} b^{-1/\kappa^2} k^{q+1}.
\end{equation*}
We have
\begin{equation*}
    \dashint_{B(x_0,r)\times(t_0-T,t_0)} (u-k)_+ \dnu \leq c\dashint_{Q^-_{r_n}(x_0,t_0)} (u-k)_+ \dnu \leq c \left(\frac{T}{r^p}\right)^pY_n^p,
\end{equation*}
for every $n\in\mathbb N$. Hence, for $k=cY_0+ \left( \frac{r^p}{T} \right)^\frac{1}{p-q-1}$, with $c$ large enough, the convergence $Y_n\to$ as $n\to\infty$ implies
\begin{equation*}
    \dashint_{B(x_0,r)\times(t_0-T,t_0)} (u-k)_+ \dnu = 0.
\end{equation*}
It follows that 
\begin{equation*}
    \esssup_{B(x_0,r)\times(t_0-T/2,t_0)} u \leq k = c \biggl(\frac{T}{r^p}\dashint_{B(x_0,r)\times(t_0-T,t_0)}u_+^p\dnu\biggr)^\frac{1}{q+1} + \left( \frac{r^p}{T} \right)^\frac{1}{p-q-1}.
\end{equation*}
\end{proof}

\subsection{Propagation of positivity in measure}
The following lemma shows that densities of level sets propagate in a time interval depending on the considered level. 
\begin{lemma}
\label{lemma: propagation of pos in meas}
Let $1<p<\infty$ and $0<q<\infty$. Assume that $\mu$ is doubling and that a $(1,p)$-Poincaré inequality holds in $(X,d,\mu)$. Let $u$ be a nonnegative supersolution in $\Omega_T$, $k>0$, $0<\alpha<1$ and $\sigma >1$. Then there exist constants $0<\delta_0<1$ and $0<\epsilon<1$ depending on $\alpha$, $\sigma$, $p$, $q$, and $c_\mu$, such that if $B(x_0,\sigma r)\times[t_0,t_0+\delta k^{q+1-p} r^p)\subset\Omega_T$, $0<\delta\leq \delta_0$, and $t_0$ is a Lebesgue point of $u$ with 
\begin{equation*}
    \mu( \{ u(\cdot,t_0)\geq k\}\cap B(x_0,r) ) \geq \alpha \mu(B(x_0,r)),
\end{equation*}
then 
\begin{equation*}
    \mu( \{ u(\cdot,t)\geq \epsilon k\}\cap B(x_0,r) ) \geq \frac{\alpha}{2} \mu(B(x_0,r)),
\end{equation*}
for almost every $t_0<t<t_0+\delta k^{q+1-p} r^p$.
\end{lemma}
\begin{proof}
Let $\eta>0$ with $1+\eta<\sigma$. Applying \Cref{lemma: caccioppoli with cutoff} we have
\begin{equation}
\label{eq: first est in prop of pos}
\begin{split}
    \int_{B(x_0,r)} G_{-}(u(t),k)\dmu & \leq \int_{B(x_0,{(1+\eta)r})}G_{-}(u(t_0),k)\dmu \\
    & \qquad + \frac{c}{\eta^p r^p}\int_{B(x_0,(1+\eta)r)\times (t_0, t_0 + \delta k^{q+1-p}r^p)} (u-k)_-^p\dmu\dt.
\end{split}
\end{equation}
We begin by estimating the  first term on the right hand side of the above. Splitting the integral into two parts, we obtain
\begin{equation*}
\begin{split}
    & \int_{B(x_0,{(1+\eta)r})}G_{-}(u(t_0),k)\dmu \\
    & \leq q k^{q+1} \mu(B(x_0,(1+\eta)r)\backslash \mu(B(x_0,r))) \\
    & \quad + q\int_{0}^{k} s^{q-1}(k-s)\ds\mu(\{u(t_0)<k\}\cap B(x_0,r)) \\
    & \leq c k^{q+1} \eta^\beta \mu(B(x_0,r)) + q\int_{0}^{k} s^{q-1}(k-s)\ds (1-\alpha)\mu(B(x_0,r)).
\end{split}
\end{equation*}
Here, we used \eqref{eq: geodesic doubling becomes annular decay} in the second inequality. We have that
\begin{equation*}
    q\int_{0}^{k} s^{q-1}(k-s)\ds \leq q \int_{\epsilon k}^{k} s^{q-1}(k-s)\ds + c(\epsilon k)^{q+1},
\end{equation*}
and hence 
\begin{equation*}
\begin{split}
    & \int_{B(x_0,{(1+\eta)r})}G_{-}(u(t_0),k)\dmu \\
    & \leq c k^{q+1} \eta^\beta \mu(B(x_0,r)) +  q \int_{\epsilon k}^{k} s^{q-1}(k-s)\ds(1-\alpha)\mu(B(x_0,r)) \\
    & \qquad + c(\epsilon k)^{q+1}\mu(B(x_0,r)).
\end{split}
\end{equation*}
We estimate the second term on the right hand side of \eqref{eq: first est in prop of pos} by
\begin{equation*}
    \frac{c}{\eta^p r^p}\int_{B(x_0,(1+\eta)r)\times (t_0, t_0 + \delta k^{q+1-p}r^p)} (u-k)_-^p\dmu\dt \leq \frac{c \delta_0}{\eta^p r^p} k^{q+1} \mu(B(x_0,r)).
\end{equation*}
The left hand side of \eqref{eq: first est in prop of pos} is estimated from below with 
\begin{equation*}
    \int_{B(x_0,r)} G_{-}(u(t),k)\dmu \geq q\int_{\epsilon k}^k s^{q-1}(k-s)\ds \mu(\{u(t)<\epsilon k\}\cap B(x_0,r)).
\end{equation*}
Finally, we observe that
\begin{equation*}
    q\int_{\epsilon k}^k s^{q-1}(k-s)\ds \geq \gamma k
\end{equation*}
for a constant $\gamma$ depending only on $q$ provided that $\epsilon \leq \frac{1}{4}$. Putting together the estimates we have obtained thus far we have 
\begin{equation*}
    \mu(\{u(t)<\epsilon k\}\cap B(x_0,r)) \leq \mu(B(x_0,r)) \left( (1-\alpha) + c\epsilon^p+ c \eta^\beta + \frac{c \delta_0}{\eta^p} \right).
\end{equation*}
Choosing $\epsilon$, $\eta$, and finally $\delta_0$ to be small enough, the previous estimate gives 
\begin{equation*}
    \mu(\{u(t)<\epsilon k\}\cap B(x_0,r)) \leq \Big(1-\frac{\alpha}{2}\Big)\mu(B(x_0,r)).
\end{equation*}
\end{proof}
We next show that densities of level sets can be improved.
\begin{lemma}
\label{lemma: shrinking lemma slow diffusion}
Let $1<p<\infty$ and $0<q<\infty$ with $q+1<p$. Assume that $\mu$ is doubling and that a $(1,p)$-Poincaré inequality holds in $(X,d,\mu)$. Let $u$ be a nonnegative supersolution in $\Omega_T$, $k>0$, $\alpha\in (0,1)$, $\sigma>1$ and  $0<\gamma<1$. Let $0<\delta\leq \delta_0$ and $\epsilon$ be as in \Cref{lemma: propagation of pos in meas}. Then there exists a constant $\lambda>0$ depending on $\alpha$, $\sigma$, $\gamma$, $\delta$, $p$, $c_\mu$, and $c_P$ such that if $B(x_0,\sigma r)\times [t_0,t_0+ \delta(\lambda k)^{q+1-p} r^p)\subset\Omega_T$ and $t_0$ is a Lebesgue point of $u$ with
\begin{equation}
    \label{eq: initial time assump in shrinking lemma}
    \mu(\{u(t_0) > k \}\cap B(x_0,r)) \geq \alpha \mu(B(x_0,r)),
\end{equation}
then
\begin{equation*}
    \nu\big(\big\{ u < \epsilon \lambda k\big\}\cap Q\big) \leq \gamma \nu(Q), 
\end{equation*}
where $Q=B(x_0,r)\times (t_0,t_0+\delta (\lambda k)^{q+1-p} r^p)$.
\end{lemma}
\begin{proof}
Let $k_i=\frac{k}{2^i}\eta$, where $i\in \{0,\dots,j\}$ and $j$ is to be chosen later, and let $\lambda=2^{-j}$. Denote $v=\max(u,k_i)-\max(u,k_{i+1})$. Applying \Cref{lemma: propagation of pos in meas} we find that 
\begin{equation*}
    \mu(\{u(t)>\epsilon k_i\} \cap B(x_0,r) ) \geq \mu(\{u(t) > \epsilon k_j \}\cap B(x_0,r)) \geq \frac{\alpha}{2}\mu(B(x_0,r)),
\end{equation*}
for all $t\in (t_0,t_0+\delta (\lambda k)^{q+1-p} r^p)$, where $\delta$ is as in \Cref{lemma: propagation of pos in meas}. Thus, we have
\begin{equation*}
    \mu(\{v(t) > 0\} \cap B(x_0,r) ) \leq \Big(1-\frac{\alpha}{2}\Big)\mu(B(x_0,r)),
\end{equation*}
for a.e. $t\in (t_0,t_0+\delta (\lambda k)^{q+1-p} r^p)$. For such $t$, we apply the Poincaré inequality from \Cref{eq: Poincaré for function zero on large set} with exponent $1\leq \tau<p$ on the right hand side which yields
\begin{equation*}
\begin{split}
    (k_i-k_{i+1})  \mu(\{u(\cdot,t)<k_{i+1}\}\cap B(x_0,r)) & \leq \int_{B(x_0,r)} v(t)\dmu \\
    & \leq c r \mu(B(x_0,r))^{1-\frac{1}{\tau}} \Big(\int_{B(x_0,r)}g_{v(t)}^\tau \dmu \Big)^\frac{1}{\tau}.
\end{split}
\end{equation*}
for a constant $c=c(\alpha,c_\mu,c_P,p,q)$. Integrating the above over $t\in (t_0,t_0+\delta (k_j)^{q+1-p} r^p)$ and applying Hölder's inequality yields
\begin{equation*}
\begin{split}
    (k_i-k_{i+1}) \nu(\{u<k_{i+1}\}  \cap Q) & \leq c r \mu(B(x_0,r))^{1-\frac{1}{\tau}} \int_{t_0}^{t_0+\delta r^p}\Big(\int_{B(x_0, r)}g_{v(t)}^\tau \dmu \Big)^\frac{1}{\tau} \dt \\
    & \leq cr \nu(Q)^{1-\frac{1}{\tau}} \left( \int_{Q} g_v^\tau \dmu\dt \right)^\frac{1}{\tau}.
\end{split}
\end{equation*}
Another application of Hölder's inequality gives 
\begin{equation*}
\begin{split}
    &(k_i-k_{i+1}) \nu(\{u<k_{i+1}\} \cap Q) \\
    & \leq cr \nu(Q)^{1-\frac{1}{\tau}} \left( \int_{Q} g_{(u-k_{i})_-}^p \dmu\dt \right)^\frac{1}{p} \nu(\{ k_{i+1}<u<k_i\}\cap Q )^{\frac{1}{\tau}-\frac{1}{p}}.
\end{split}
\end{equation*}
Here, we used the fact that $g_v=g_u \chi_{\{k_{i+1} < u <k_i \}}$. Applying \Cref{lemma: caccioppoli with cutoff} we obtain
\begin{equation*}
\begin{split}
    \int_{Q} g_{(u-k_i)_-}^p \dmu\dt & \leq \int_{B(x_0,\sigma r)\times (t_0,t_0+ c_1 k^{q+1}r^p)} \frac{1}{r^p} (u-k_i)^p_-\dmu\dt + c\int_{B(x_0,\sigma r)} G_{-}(u(t_0),k_i)\dmu \\
    & \leq c r^{-p} \nu(Q) k_i^p + c \mu (B(x_0,r)) k_i^{q+1} \\
    & = c r^{-p} \nu(Q) k_i^p + \nu(Q) \frac{1}{(\delta \lambda k)^{q+1-p}}k_i^{q+1}\\
    & \leq c r^{-p} \nu(Q)k_i^p.
\end{split}
\end{equation*}
for a constant $c=c(\alpha,\delta,c_\mu,c_P,p,q)$. The previous estimates imply 
\begin{equation*}
    \nu(\{u<k_{j+1}\} \cap Q) \leq c \nu(Q)^{1-\frac{1}{\tau}+\frac{1}{p}}  \nu(\{ k_{i+1}<u<k_i\}\cap Q )^{\frac{1}{q}-\frac{1}{p}}.
\end{equation*}
Taking powers, we deduce that 
\begin{equation*}
    \nu(\{u<k_{i+1}\} \cap Q)^{\frac{p\tau}{p-\tau}} \leq  c \nu(Q)^{\frac{p\tau}{p-\tau}-1}  \nu(\{ k_{i+1}<u<k_i\}\cap Q).
\end{equation*}
Summing the previous inequality over $i\in \{0,\dots,j\}$ gives
\begin{equation*}
    j \nu(\{u<k_{j}\} \cap Q)^{\frac{p\tau}{p-\tau}} \leq  c \nu(Q)^{\frac{p\tau}{p-\tau}},
\end{equation*}
and hence 
\begin{equation*}
    \nu(\{u<k_{j}\} \cap Q) \leq  c j^{-\frac{p-\tau}{p\tau}} \nu(Q).
\end{equation*}
The claim follows by choosing $j$ large enough so that $c j^{-\frac{p-\tau}{p\tau}}\leq \gamma$.
\end{proof}

\subsection{A De Giorgi--Di Benedetto lemma}
The following result is a De Giorgi--Di Benedetto lemma in the intrinsic geometry of the equation. 
\begin{lemma}
\label{lemma: de giorgi lemma}
Let $1<p<\infty$ and $0<q<\infty$. Assume that $\mu$ is doubling and that a $(1,p)$-Poincaré inequality holds in $(X,d,\mu)$. Let $k>0$, $\delta>0$, $\sigma>1$, $\mu^-\in \mathbb R$, and assume that $u$ is a supersolution in $Q^-_{\sigma r,\theta}(x_0,t_0)$, where $\theta=\delta k^{q+1-p}$, $\sigma r\leq \diam(X)/3$ and
\begin{equation*}
    u \geq \mu^- \text{ a.e. in } Q^-_{\sigma r,\theta}(x_0,t_0).
\end{equation*}
Let $a\in (0,1)$ and $M>1$. There exists a constant $\gamma\in (0,1)$ depending on $c_\mu, c_P$, $p$, $q$, $\delta$, $\sigma$, $a$ and  $M$, such that if 
\begin{equation}
    \label{eq: meas dens cond in de giorgi lemma}
    \nu(\{ u -\mu^-< k  \}\cap Q^-_{\sigma r,\theta}(x_0,t_0)) \leq \gamma \nu(Q^-_{\sigma r,\theta}(x_0,t_0))
\end{equation}
and 
\begin{equation}
    \label{eq: assumption for de giorgi lemma}
    \lvert \mu^-\rvert\leq M k,
\end{equation}
then 
\begin{equation*}
    u -\mu^-\geq a k,
\end{equation*}
a.e. in $Q^-_{ r,\theta}(x_0,t_0)$.
\end{lemma}
\begin{proof}

Let $n\in \mathbb N$ and 
\begin{equation*}
      k_n = \mu^-+ak+\frac{1-a}{2^n}k, \quad \Tilde k_n=\frac{k_{n+1}+k_n}{2}, \quad r_n = r + \frac{\sigma-1}{2^{n}} r,
\end{equation*}
$B_n=B(x_0,r_n)$ and $Q_n=Q^-_{r_n,\theta}(x_0,t_0)$. Let $\eta_n$ be a $\frac{c}{r_n-r_{n+1}}$-Lipschitz cut-off function with $0\leq \eta\leq 1$, $\eta=1$ on $B_{n+1}$ and $\supp \eta \subset B_n$. Let
\begin{equation*}
    \xi_n(t)
        =\begin{cases}
            0 ,&\quad t\leq t_0-\theta r_{n}^p\\
            \frac{1}{\theta(r_n^p-r_{n+1}^p)}(t-t_0+\theta r_{n}^p), & \quad t_0-\theta r_n^p <t\leq t_0-\theta r_{n+1}^p \\
            1,&\quad t\geq t_0-\theta r_{n+1}^p,
        \end{cases}
\end{equation*}
and denote $\varphi_n=\eta_n\xi_n$. Let $\kappa$ be as in \Cref{lemma: par sob}. Applying \Cref{lemma: par sob} gives
\begin{equation}
\label{eq: de giorgi lemma est from par sob}
\begin{split}
    k^{p+\kappa(q+1)}\frac{\nu(A_{n+1})}{\nu(Q_{n+1})} & \leq c2^{n(p+\kappa(q+1))}\dashint_{Q_n}(u-\Tilde k_n)_-^{p+\kappa(q+1)}\varphi_n\dmu\dt \\
    & \leq c 2^{n(p+\kappa(q+1))} r_n^p\dashint_{Q_n}g_{(u-\Tilde k_n)_-\varphi_n}^p\dmu\dt \\
    & \qquad \cdot \esssup_{t_0-\theta r_n^p<t<t_0}\left(\dashint_{B_n}(u(t)-\Tilde k_n)_-^{q+1} \varphi_n^p \dmu\right)^{\kappa}.
\end{split}
\end{equation}
Let us estimate the right hand side of the above. Using \Cref{lemma: caccioppoli with cutoff} we have 
\begin{equation}
\label{eq: cacciop for de giorgi lemma}
\begin{split}
     \esssup_{t_0-\theta r_n^p<t<t_0}& \dashint_{B_n} (\lvert u(t) \rvert +\lvert k_n\rvert )^{q-1}(u(t)-k_n)_-^2 \varphi_n^p\dmu + \theta r_n^p \dashint_{Q_n}g_{(u-k_n)_-}^p \varphi_n^p\dmu\dt \\
    & \leq c \theta 2^{pn}\dashint_{Q_n}(u-k_n)_-^p +\theta^{-1}( \lvert u\rvert +\lvert k_n\rvert )^{q-1}(u-k_n)_-^2\dmu\dt \\
    & \leq c 2^{pn} k^{q+1} \frac{\nu(A_n)}{\nu(Q_n)},
\end{split}
\end{equation}
where in the last step we used \eqref{eq: assumption for de giorgi lemma}, and the definition of $\theta$. Let us compare $(u-\Tilde k_n)_-^{q+1}$ and $(\lvert u \rvert +\lvert k_n\rvert )^{q-1}(u-k_n)_-^2$. Assume first that $0<q<1$. Using \eqref{eq: assumption for de giorgi lemma} we have $\lvert u\rvert + \lvert k_n\rvert \leq c k$ on $\{u<\Tilde k_n\}$, and thus 
\begin{equation*}
\begin{split}
    (u-\Tilde k_n)_-^{q+1}\leq k^{q+1}\chi_{u<\Tilde k_n} & \leq c (\lvert u\rvert +\lvert k_n\rvert )^{q-1} k^2\chi_{u<\Tilde k_n} \\
    &\leq c 2^{2n}c (\lvert u\rvert +\lvert k_n\rvert )^{q-1} (u-k_n)_-^2.
\end{split}
\end{equation*}
For $q\leq 1 <\infty$, we simply estimate
\begin{equation*}
    (u-\Tilde k_n)_-^{q+1}\leq (u-k_n)_-^{q+1}\leq  (\lvert u\rvert +\lvert k_n\rvert )^{q-1} (u-k_n)_-^2.
\end{equation*}
In either case we conclude that 
\begin{equation*}
\begin{split}
    &\esssup_{t_0-\theta r_n^p<t<t_0}\dashint_{B_n} (u(t)-\Tilde k_n)_+^{q+1} \varphi_n^p\dmu\\
    & \leq c2^{2n}\esssup_{t_0-\theta r_n^p<t<t_0} \dashint_{B_n} (\lvert u(t) \rvert +\lvert k_n\rvert )^{q-1}(u(t)-k_n)_-^2 \varphi_n^p\dmu.
\end{split}
\end{equation*}
The previous estimate, \eqref{eq: cacciop for de giorgi lemma} and the product rule imply
\begin{equation}
\label{eq: wanted est for de giorgi lemma}
\begin{split}
    & \esssup_{t_0-\theta r_n^p<t<t_0}\dashint_{B_n} (u(t)-\Tilde k_n)_-^{q+1} \varphi_n^p\dmu + \theta r_n^p\dashint_{Q_n}g_{(u-\Tilde k_n)_- \varphi_n}^p\dmu\dt  \\
    & \leq  c2^{2n}\esssup_{t_0-\theta r_n^p<t<t_0} \dashint_{B_n} (\lvert u(t) \rvert +\lvert k_n\rvert )^{q-1}(u-k_n)_-^2 \varphi_n^p\dmu + \\
    & \qquad + c \theta r_n^p\dashint_{Q_n}g_{(u-k_n)_-}^p \varphi_n^p\dmu\dt + c\theta r_n^p\dashint_{Q_n}(u-k_n)_-^p g_{\varphi_n}^p\dmu\dt\\
    & \leq c 2^{(p+2)n} k^{q+1} \frac{\nu(A_n)}{\nu(Q_n)}.
\end{split}
\end{equation}
Applying \eqref{eq: wanted est for de giorgi lemma} in \eqref{eq: de giorgi lemma est from par sob} we find that 
\begin{equation}
    \label{eq: decay in de giorgi lemma proof}
    \frac{\nu(A_{n+1})}{\nu(Q_{n+1})}\leq c2^{((p+2)(1+\kappa)+p+\kappa(q+1))n}\left( \frac{\nu(A_n)}{\nu(Q_n)} \right)^{1+\kappa}.
\end{equation}
Denote $Y_n=\frac{\nu(A_{n+1})}{\nu(Q_{n+1})}$. Using \cite[Lemma 5.1]{DiBenedettoGIanazzaVespri2012} to \eqref{eq: decay in de giorgi lemma proof} we conclude that $Y_n\to 0$ as $n\to \infty$ provided that \eqref{eq: meas dens cond in de giorgi lemma} holds with $\gamma$ small enough. Since
\begin{equation*}
    \frac{\nu(\{u<\mu^-+ak\}\cap Q_{r,\theta}^-(x_0,t_0))}{\nu(Q_{ r,\theta}^-(x_0,t_0))} \leq c Y_n,
\end{equation*}
for all $n\in N$, the convergence $Y_n\to 0$ as $n\to\infty$ implies $\essinf_{Q_{r,\theta}^-(x_0,t_0)} u \geq \mu^-+ak$.
\end{proof}

\begin{remark}
\label{remark: resctrion on radius in de giorgi lemma can be removed}
The restriction $\sigma r\leq \diam(X)/3$ in \Cref{lemma: de giorgi lemma} can be removed. Thus, let $\delta>0$, $\sigma>1$, $\mu^-\in \mathbb R$, and assume that $u$ is a supersolution in $Q^-_{\sigma r,\theta}(x_0,t_0)$, where
\begin{equation*}
    u \geq \mu^- \text{ a.e. in } Q_{\sigma r,\theta}^-(x_0,t_0).
\end{equation*}
Let $0<a<1$ and $M>1$. We may assume $\sigma \leq 2$. If $r\leq \diam(X)/6$, we simply apply \Cref{lemma: de giorgi lemma}. Thus, we may assume that $\diam(X)/6\leq r \leq \diam(X)$. The doubling condition implies that there exists $0<\tau<1$ depending only on $c_\mu$, $p$ and $\sigma$, such that if 
\begin{equation*}
    \nu(\{ u -\mu^-< k  \}\cap Q_{\sigma r,\theta}^-(x_0,t_0)) \leq \tau \gamma \nu(Q_{\sigma r,\theta}^-(x_0,t_0))
\end{equation*}
then 
\begin{equation}
    \label{eq: meas dens cond for applying de giorgi in smaller cylinder}
    \nu(\{ u -\mu^-< k  \}\cap Q_{(\sigma -1) r/3,\theta}^-(z)) \leq  \gamma \nu(Q_{(\sigma-1)r/3,\theta}^-(z))
\end{equation}
for every $z\in Q_{r,\theta}(x_0,t_0)$. Since $(\sigma-1) r/3\leq \diam(X)/3$, we use \eqref{eq: meas dens cond for applying de giorgi in smaller cylinder} and apply \Cref{lemma: de giorgi lemma} in $Q_{(\sigma-1)r/3,\theta}^-(z)$ to conclude that $u -\mu^-\geq a k$ a.e. in $Q_{(\sigma-1)r/6,\theta}^-(z)$. Since $z$ was arbitrary, we conclude $u -\mu^-\geq a k$
a.e. in $Q_{r,\theta}^-(x_0,t_0)$.
\end{remark}

Arguing as in \cite{Liao2021}, the following corollary is implied by \Cref{lemma: de giorgi lemma}. We include the proof for the convenience of the reader.
\begin{corollary}
\label{cor: lower semicont of supersolutions}
Let $1<p<\infty$ and $0<q<\infty$ with $q+1<p$. Assume that $\mu$ is doubling and that a $(1,p)$-Poincaré inequality holds in $(X,d,\mu)$.
Let $u$ be a supersolution of \eqref{eq: PDE} in $\Omega_T$. Then the lower semicontinuous function 
\begin{equation*}
    u^*(z)=\lim_{r\to 0} \essinf_{Q_r^-(z)} u
\end{equation*}
agrees with $u$ $\nu$-almost everywhere in $\Omega_T$. The same conclusion holds for $p<q+1$ if $u$ is locally bounded from below.
\end{corollary}
\begin{proof}
By \Cref{prop: esssup estimate}, a supersolution is locally bounded from below. Let $z\in\Omega_T$ be a Lebesgue point of $u$. Clearly $u^*(z)\leq u(z)$. If $u(z)\leq u^*(z)$ we are done, so assume that $u^*(z)<u(z)$. Let $R>0$ be such that $Q^-_{R}(z)\Subset\Omega_T$ and $\mu^-=\essinf_{Q^-_{R}(z)} u$. Since $\mu^-\leq u^*(z)<u(z)$ there exist $k>0$ and $0<a<1$ such that
\begin{equation*}
    \mu^- + k < u(z) 
    \quad\text{and}\quad 
    \mu^-+ a k > u^*(z).
\end{equation*}
Denote $M=\frac{\lvert \mu^-\rvert}{k}+1$. Let $\gamma$ be as in \Cref{lemma: de giorgi lemma} corresponding to the parameters $\delta=k^{p-q-1}$, $\sigma=2$, $a$ and $M$. We claim that 
\begin{equation}
    \label{eq: density for applying de giorgi lemma for lower semicont}
    \nu(\{u<\mu^-+k\}\cap Q^-_{\rho}(z))<\gamma \nu(Q^-_\rho(z)),
\end{equation}
for some $0<\rho<R/2$. If \eqref{eq: density for applying de giorgi lemma for lower semicont} does not hold, then 
\begin{equation*}
\begin{split}
    \int_{Q^-_{\rho}(z)} \lvert u(z)-u\rvert \dnu 
    & \geq (u(z)-(\mu^-+k) \nu(\{u< \mu^-+k\}\cap Q^-_{\rho}(z)) \\
    & \geq (u(z)-(\mu^-+k))\gamma \nu(Q^-_\rho(z)),
\end{split}
\end{equation*}
for every $0<\rho<R$. Letting $\rho\to 0$ this contradicts the assumption that $z$ is a Lebesgue point. Hence, we conclude that \eqref{eq: density for applying de giorgi lemma for lower semicont} holds for some $0<\rho<R/2$ and apply \Cref{lemma: de giorgi lemma} to conclude that 
\begin{equation*}
    u \geq \mu^-+ ak>u^*(z) 
\end{equation*}
$\nu$-almost everywhere in  $Q^-_{\frac{\rho}{2}}(z)$.
This contradicts the definition of $u^*(z)$, and hence we conclude that $u(z)\leq u^*(z)$.
\end{proof}

\subsection{Expansion of positivity}
In the following result, we obtain an expansion of positivity result in an intrinsic geometry.
\begin{proposition}
\label{prop: expansion of pos}
Let $1<p<\infty$ and $0<q<\infty$ with $q+1<p$. Assume that $\mu$ is doubling and that a $(1,p)$-Poincaré inequality holds in $(X,d,\mu)$. Let $u$ be a nonnegative supersolution in $\Omega_T$, $k>0$, $\alpha\in (0,1)$, and $\sigma>1$. There exist constants $b>1$ and $\lambda>0$, depending on $\alpha$, $\sigma$, $p$, $q$, $c_\mu$ and $c_P$, such that if  $B(x_0,\sigma r)\times[t_0,t_0+b k^{q+1-p}r^p)\subset\Omega_T$ and $t_0$ is a Lebesgue point of $u$ with 
\begin{equation}
    \label{eq: assumption in exp of pos}
    \mu( \{u(t_0)<k\}\cap B(x_0,r)) \leq \alpha \mu(B(x_0,r))
\end{equation}
then $u \geq \lambda k$ almost everywhere in $B(x_0,r)\times(t_0+2^{-p} b k^{q+1-p} r^p,t_0 b k^{q+1-p}r^p)$.
\end{proposition}
\begin{proof}
Let $h>0$ and let $1<\tau \leq \min(\sqrt{\sigma},2)$ with $\frac{1}{\tau^p}\geq (1-2^{-p})$. Let $\epsilon$ and $\delta_0$ be as in \Cref{lemma: propagation of pos in meas}. Applying \Cref{lemma: de giorgi lemma} with parameters  $\delta_0/\tau^p$, $\tau$, $a=\frac{1}{2}$, there exists a constant $\gamma$ depending on $c_\mu$, $ c_P$, $p$, $q$, and $\sigma$ such that if 
\begin{equation}
\begin{split}
    \nu(\{ & u < \epsilon h \}\cap B(x_0,\tau r)\times(t_0,t_0+\delta_0 h^{q+1-p}r^p)) \\
    &\leq \gamma \nu(B(x_0,\tau r)\times(t_0,t_0+\delta_0 h^{q+1-p}r^p)),
\end{split}
\end{equation}
then 
\begin{equation}
    \label{eq: applying de giorgi lemma for exp of pos}
    u \geq \frac{1}{2}h \text{ a.e. in } B(x_0,r)\times(t_0+2^{-p}\delta_0(h)^{q+1-p} r^p,t_0+\delta_0 h^{q+1-p} r^p).
\end{equation}
On the other hand, the assumption \eqref{eq: assumption in exp of pos} implies that
\begin{equation*}
    \mu(\{u(t_0)<k\}\cap B(x_0,\tau r))\leq \Big(1-\frac{1-\alpha}{c_\mu}\Big)\mu(B(x_0,\tau r)),
\end{equation*}
and hence \Cref{lemma: shrinking lemma slow diffusion} gives 
\begin{equation}
\begin{split}
    \nu(\{ u & < \epsilon \lambda k  \}\cap B(x_0,\tau r)\times(t_0,t_0+\delta_0 (\lambda k)^{q+1-p}r^p)) \\
    &\leq \gamma \nu(B(x_0,\tau r)\times(t_0,t_0+\delta_0 (\lambda k)^{q+1-p}r^p)),
\end{split}
\end{equation}
for some $\lambda$ as in the statement. Applying \eqref{eq: applying de giorgi lemma for exp of pos} with $h=\lambda k$ shows the claim.
\end{proof}
\subsection{Parabolic Harnack inequality}
We now combine the previous results to show the intrinsic parabolic Harnack inequality.
\begin{theorem}
\label{thm: PI implies PHI}
Let $1<p<\infty$ and $0<q<\infty$ with $q+1<p$. Assume that $\mu$ is doubling and that a $(1,p)$-Poincar\'e inequality holds in $(X,d,\mu)$. Let $u$ be a nonnegative solution to \eqref{eq: PDE} in $\Omega_T$. Then, there exist constants $c_1,c_2>0$ depending on $p$, $q$, $c_\mu$ and $c_P$ such that for almost every $(x_0,t_0)\in \Omega_T$, if $u(x_0,t_0)>0$, $r>0$ and 
$B(x_0,4r)\times (t_0-c_1 u(x_0,t_0)^{q+1-p}(4r)^p,t_0+c_1 u(x_0,t_0)^{q+1-p}(4r)^p)\subset \Omega_T$, we have 
\begin{equation*}
    u(x_0,t_0) \leq c_2 \essinf_{Q^+} u,
\end{equation*}
where $Q^+=B(x_0,r)\times (t_0+c_1 u(x_0,t_0)^{q+1-p}/2 r^p,t_0+c_1 u(x_0,t_0)^{q+1-p} r^p)$.
\end{theorem}
\begin{proof}
Denote $z_0=(x_0,t_0)$. If $u(z_0)=0$ we are done. Otherwise, denote $\omega=u(z_0)^{q+1-p}$, and let
\begin{equation*}
    m(\rho) = \esssup_{Q^{-}_{\rho,\omega}(z_0)} u \quad\text{and}\quad n(\rho)=\biggl(1-\frac{4\rho}{r}\biggr)^{-\beta}u(z_0)    
\end{equation*}
for $\rho\in [0,r/4)$, and some $\beta>0$. By \Cref{cor: lower semicont of supersolutions}, we may assume that $m(0)=u(z_0)=n(0)$. Let
\begin{equation*}
    r_0 = \sup \{ \rho \in (0,r): m(\rho)\geq n(\rho) \}.
\end{equation*}
Note that $r_0<r/4$ since $m$ is bounded and $n(\rho)\to\infty$ as $\rho\to r/4$. By upper semicontinuity of $m-n$ we have $m(r_0)\geq n(r_0)$. There must exist $\Tilde z=(\Tilde x, \Tilde t)\in Q_{r_0,\theta}^-(z_0)$ such that 
\begin{equation}
    \label{eq: sup almost realized for exp of pos proced}
    u(\Tilde z)\geq \frac{3}{4}m(r_0) \geq \frac{3}{4}n(r_0).
\end{equation}
Let $\tau=\frac{5}{4}$ and let $n\in \mathbb N$ such that $\tau^{-n}\leq \frac{\frac{r}{4}-r_0}{2r}\leq \tau^{1-n}$. Define $\rho=\tau^{-n}r.$ Then $Q^-_{\rho,\theta}(\Tilde z)\subset Q^-_{\frac{r_0+r/4}{2},\theta}(z_0)$  and by definition of $r_0$ we obtain
\begin{equation*}
\begin{split}
    \esssup_{Q_{\rho,\theta}(\Tilde z)} u & \leq \esssup_{Q_{\frac{r/4+r_0}{2},\theta}(z)} u \leq n\biggl(\frac{r/4+r_0}{2}\biggr) \\
    &= 2^{-\beta}\biggl(\frac{\frac{r}{4}-r_0}{2r}\biggr)^{-\beta}u(z_0) \\
    & \leq 2^{-\beta} \tau^{n\beta} u(z_0).
\end{split}
\end{equation*}
Set $\mu=2^{-\beta} \tau^{n\beta} u(z_0)$. We claim that 
\begin{equation}
    \label{eq: measure density dep on beta}
    \nu( \{ u > \frac{1}{4}\tau^{-\beta} \mu \}\cap Q_{\rho,\omega}^-(\Tilde z)) \geq \gamma_0\nu(Q^-_{\rho,\omega}(\Tilde z))
\end{equation}
for some $\gamma_0$ depending only on $c_\mu$, $c_P$, $p$, $q$ and $\beta$. If \eqref{eq: measure density dep on beta} fails, then we have
\begin{equation*}
    \nu(\{ \mu-u < \mu-\frac{1}{4}\tau^{-\beta}\mu\}\cap Q^-_{\rho,\omega}(\Tilde z)) \leq \gamma_0 \nu(Q_{\rho,\omega}^-(\Tilde z)).
\end{equation*}
Applying \Cref{lemma: de giorgi lemma} to the supersolution $-u$, with parameters 
\begin{equation*}
    \mu^-=-\mu, \quad a=\frac{1-\frac{\tau^{-\beta}}{2}}{1-\frac{\tau^{-\beta}}{4}}, \quad \sigma=2\quad\text{and } M=\frac{1}{1-\frac{1}{4}\tau^{-\beta}}
\end{equation*}
we find that 
\begin{equation*}
    \mu-u(\Tilde z) \geq a( \mu-\frac{1}{4}\tau^{-\beta}\mu)=\mu-\frac{1}{2}\tau^{-\beta} \mu,
\end{equation*}
if $\gamma_0$ is small enough, depending only on $c_\mu$, $c_P$, $p$, $q$ and $\beta$. The previous estimate can be equivalently rewritten as 
\begin{equation*}
    u(\Tilde z) \leq \frac{1}{2}\tau^{-\beta}\mu.
\end{equation*}
This can not hold, since \eqref{eq: sup almost realized for exp of pos proced} implies $u(\Tilde z)\geq \frac{3}{4}\tau^{-\beta}\mu^{+}$. Hence, we conclude that \eqref{eq: measure density dep on beta} holds. Thus, there exists $t'\in (\Tilde t-\omega \rho^p, \Tilde t)$ with 
\begin{equation*}
    \mu( \{ u(t') > \frac{1}{4} \tau^{-\beta} \mu \}\cap B(x_0,\rho)) \geq \gamma_0 \mu(B(x_0,\rho)).
\end{equation*}
Applying \Cref{prop: expansion of pos} we obtain 
\begin{equation}
    \label{eq: first expanded pos}
    \essinf_{B(\Tilde x,\rho))}u(s_0) \geq \lambda_0 \mu 
\end{equation}
for some $\lambda_0$ and $b_0$ depending on $c_\mu, c_P, p$ and $\beta$, and $s_0=t'+b_0 \rho^p (n(r_0))^{q+1-p}$. Then 
\begin{equation*}
    \mu(\{u(s_0) > \lambda_0 \mu \} \cap B(\Tilde x, \tau\rho)) \geq \frac{1}{c_\mu} \mu(B(\Tilde x, \tau\rho)).
\end{equation*}
Let $\lambda$ and $c_1$ be the constants in \Cref{prop: expansion of pos} corresponding to $\alpha=(1-\frac{1}{c_\mu})$ and $\sigma=\tau$. For $i\in \{1,2,3,\dots\}$, we let 
\begin{equation*}
    s_i=s_{i-1}+b (\tau^n\rho)^p(\lambda_0\lambda^{n-1}\mu)^{q+1-p}.
\end{equation*}
Applying \Cref{prop: expansion of pos} we find that 
\begin{equation*}
    \essinf_{B(\Tilde x,\tau \rho)}u(s_1) \geq \lambda \lambda_0 \mu
\end{equation*}
for some $\lambda$ and $b$ depending on $c_\mu$, $c_p$, $p$ and $q$. Hence
\begin{equation*}
    \mu(\{u(s_1)>\lambda \lambda_0\mu\}\cap B(\Tilde x,\tau^2\rho))\geq \frac{1}{c_\mu}\mu(B(\Tilde x,\tau^2\rho)).
\end{equation*}
Continuing iteratively, we find that 
\begin{equation}
\label{eq: expanded pos for Harnack after iteration}
    \essinf_{B(\Tilde x, \tau^{n}\rho)} u(s_n) \geq \lambda^{n} \lambda_0 \mu,
\end{equation}
provided that $s_n\leq t_0+\theta (4r)^p$. We will see that $s_\leq t_0+\theta (4r)^p$ holds if the constant $c_1$ is large enough. Let $0<\delta<1$. The previous estimate implies that
\begin{equation*}
    u(s_n) \geq \delta \lambda^{n} \lambda_0 \mu\text{ in } B(\Tilde x, \tau^{n}\rho).
\end{equation*}
and hence arguing as before and applying \Cref{prop: expansion of pos} once more we obtain 
\begin{equation*}
    u(s_n+b (\tau r)^p(\delta\lambda^n\lambda_0\mu)^{q+1-p}) \geq \delta \lambda^{n} \lambda_0 \mu\text{ in } B(\Tilde x, \tau^{n+1}\rho).
\end{equation*}
We recall $\mu=2^{-\beta} \tau^{n\beta} u(z_0)$, and choose $\beta=-\log_{\tau}\lambda$. Thus $\lambda^n \mu=2^{-\beta} u(z_0)$. By the definition of $\rho$ and $\tau$ and using the fact that $d(x_0,\Tilde x)\leq \frac{r}{4}$ the previous estimate yields
\begin{equation}
    \label{eq: Harnack before choice of delta to get inf time}
    u(s_n+b (\tau r)^p(\delta\lambda_0 2^{-\beta} u(z_0))^{q+1-p}) \geq \delta 2^{-\beta} \lambda_0 u(z_0)\text{ in } B(x_0, \rho).
\end{equation}
We have
\begin{equation*}
    s_n = t' + b_0 \rho^p \mu^{q+1-p} + b \lambda_0^{q+1-p}\rho^p \mu^{q+1-p}\sum_{i=1}^n \tau^{ip}\lambda^{(i-1)(q+1-p)},
\end{equation*}
and note that 
\begin{equation}
    \label{eq: bounds for sn harnack proof}
    t_0 - 2 r^pu(z_0)^{q+1-p}\leq s_n \leq t_0+ 2^{\beta(p-q-1)}(b_0 + b\tau^p\lambda_0^{q+1-p})r^p u(z_0)^{q+1-p} 
\end{equation}
Let $c_1>0$ to be determined, and let 
\begin{equation*}
    t_m\in (t_0+c_1 u(x_0,t_0)^{q+1-p}/2 r^p,t_0+c_1 u(x_0,t_0)^{q+1-p} r^p)
\end{equation*} be such that 
\begin{equation*}
    \essinf_{B(x_0,r)}u(t_m)\leq 2\essinf_{Q^+} u. 
\end{equation*}
We claim that for $c_1$ sufficiently large as in the statement of the theorem, we can choose $\delta$ depending only on such that $s_n+b (\tau r)^p (\delta\lambda_0 2^{-\beta} u(z_0))^{q+1-p}=t_m$. We set
\begin{equation}
    \label{eq: choice of delta in harnack proof}
    \delta^{q+1-p} = \frac{t_m-s_n}{(\tau r)^pb\lambda_0^{q+1-p}2^{-\beta}u(z_0)^{q+1-p}}.
\end{equation}
By \eqref{eq: bounds for sn harnack proof} we have
\begin{equation*}
\begin{split}
    \frac{c_1/2-2^{\beta(p-q-1)} (b_0 + b\tau^p\lambda_0^{q+1-p})}{\tau^p b\lambda_0^{q+1-p}2^{-\beta}}& \leq \frac{t_m-s_n}{(\tau r)^pb\lambda_0^{q+1-p}2^{-\beta}u(z_0)^{q+1-p}} \\
    & \leq \frac{c_1+2}{\tau^p b\lambda_0^{q+1-p}2^{-\beta}},
\end{split}
\end{equation*}
which guarantees that $\delta$ as chosen in \eqref{eq: choice of delta in harnack proof} satisfies $0<\delta<1$ and $\delta$ is bounded away from zero with an estimate depending on $p$, $q$, $c_\mu$ and $c_P$, for the choice $c_1=2 2^{\beta(p-q-1)} (b_0 + b\tau^p\lambda_0^{q+1-p})+2\tau^pb\lambda_0^{q+1-p}2^{-\beta}$. Then, returning to \eqref{eq: Harnack before choice of delta to get inf time} we have
\begin{equation*}
    u(z_0)\leq c \essinf_{B(x_0,r)}u(s_n+b (\tau r)^p(\delta\lambda_0 2^{-\beta} u(z_0))^{q+1-p}) \leq c \essinf_{B(x_0,r)}u(t_m)\leq c\essinf_{Q^+} u,
\end{equation*}
which finishes the proof.
\end{proof}

\section{Existence for a Cauchy problem}
\label{section: existence}
We consider the Cauchy problem associated with \eqref{eq: PDE}, formally written as
\begin{equation}
\label{eq: Cauchy problem}
    \begin{cases}
       \partial_t (\lvert u\rvert^{q-1} u)-\div (\lvert \nabla u\rvert^{p-2}\nabla u )=0&\mbox{in }X\times(0,T),\\
       u(x,0)= u_0(x) &\mbox{in } X.
     \end{cases}
\end{equation}
For signed functions we recall the convention \eqref{eq: signconv}.

\begin{definition}
Let $1<p<\infty$ and $0<q<\infty$. Let $u_0\in L^{q+1}(X)$. A function $u\in L^p(0,T;N^{1,p}(X))\cap C([0,T];L^{q+1}(X))$ is a  solution of \eqref{eq: Cauchy problem} if $u(0)=u_0$ and 
\begin{equation}
    \label{eq: var ineq for Neumann}
   p\int_{X_T} u^{q}\partial_t{\varphi}\dnu  + \int_{X_T} g_u^p\dnu \leq \int_{X_T} g_{u+\varphi}^p\dnu,
\end{equation}
for every $\varphi\in L^p(0,T;N^{1,p}(X))\cap L^{q+1}(X)$ with $\partial_t \varphi \in L^{q+1}(\Omega_T)$ and with compact support in time.
\end{definition}
As in \cite{EleniusKinnunen}, we work with the following compactness condition for existence. This criterion is used to handle the nonlinearity in the time derivative in existence proofs for $q\neq 1$. 
\begin{definition}
    \label{def: rellich kondrachov property}
    Let $1<p<\infty$. We say that the Rellich--Kondrachov property holds in $(X,d,\mu)$, if there exists an open cover $U_i$, $i\in I$, of $X$, such that the embedding $N^{1,p}(X)\hookrightarrow L^p(U_i)$ is compact for each $i\in I$.
\end{definition}
The Rellich--Kondrachov peroperty holds, for example, for Riemannian manifolds. Other conditions guaranteeing existence could be imposed as well.

Next we establish the existence of solutions to the Cauchy problem \eqref{eq: Cauchy problem} under the assumption that the Rellich--Kondrachov property holds in $(X,d,\mu)$ and that the Newton--Sobolev space $N^{1,p}(X)$ is reflexive. The proof largely follows the argument of \cite[Theorem 5.3]{EleniusKinnunen}. The reflexivity of $N^{1,p}(X)$ is applied to conclude strong measurability of the solution $u$ as a function $u\in L^p(0,T;N^{1,p}(X))$. We note that $N^{1,p}(X)$ is reflexive if $\mu$ is doubling, see \cite{AmbrosioColomboSimone}. As a by-product of our proof we  obtain mass conservation inequalities for solutions and their upper gradients.
\begin{theorem}\label{thm: Cauchy existence}
Let $1<p<\infty$ and $0<q<\infty$ with $q+1\leq p$. Assume that $N^{1,p}(X)$ is reflexive, and for $q \neq 1$ assume that the Rellich--Kondrachov property from \Cref{def: rellich kondrachov property} holds. Let $u_0 \in N^{1,p}(X)\cap L^{q+1}(X)\cap L^\infty(X)$ be a nonnegative function. Then there exists a nonnegative solution $u$ of \eqref{eq: Cauchy problem}, which satisfies
\begin{equation}
    \label{eq: mass conserv}
    \int_{X} u(t)^{q}\dmu \leq \int_{X} u_0^{q}\dmu
\end{equation}
and
\begin{equation}
    \label{eq: neumann problem energy decrease}
    \int_X g_{u(t)}^p\dmu \leq \int_{B} g_{u_0}^p\dmu
\end{equation}
for every $0<t\le T$.
\end{theorem}
\begin{proof}
Let $k\in\mathbb N$ and let $h_k=\frac{T}{k}$, $t_i= i h_k$ for every $i\in\{0,\dots, k\}$. 
Let $u_{k,0}=u_0$. 
We will proceed recursively.
Suppose that, for some $i\in\{1,\dots,k\}$, we have selected a nonnegative function $u_{k,{i-1}}\in N^{1,p}(X)\cap L^{q+1}(X)\cap L^\infty(X)$ with $\| u_{k,i-1}\|_\infty\leq \|u_{0}\|_\infty$.
Consider the variational integral
\begin{equation*}
    F_{k,i}(u)  = \int_X g_u^p \dmu+ \frac{p}{h_k(q+1)} \int_X \bigl(\lvert u\rvert^{q+1}  -(q+1) u_{k,i-1}^{q} u\bigr)\dmu,
\end{equation*}
where $u\in N^{1,p}(X)\cap L^{q+1}(X)$. Let us show the existence of a minimizer. Let $u_j$, $j\in \mathbb N$, be a minimizing sequence. It is easy to see that the sequence $u_j$, $j\in \mathbb N$, is bounded in $L^{q+1}(X)$ and $g_{u_j}$, $j\in\mathbb N$, is bounded in $L^p(X)$. Since $0\leq u_{k,i-1}\leq \|u_0\|_\infty$ one veries that $F_{k,i}(\min(\max(u,0),\|u_0\|_\infty))\leq F_{k,i}(u)$ for any $u$. Hence, the sequence $u_j$ may be assumed to satisfy $0\leq u_j\leq \|u_0\|_\infty$. Together, this shows that the sequence $u_j$ is bounded in $N^{1,p}(X)$ and we conclude the existence of a minimizer $u_{k,i}$ with $0\leq u_{k,i} \leq \|u_0\|_\infty$.

Let $v\in N^{1,p}(X)\cap L^{q+1}(X)$ and $h>0$. By testing the minimizing property with $u_{k,i}+h(v-u_{k,i})=(1-h) u_{k,i}+hv$, and using the convexity of minimal upper gradients, we obtain
\begin{equation*}
\begin{split}
    &\int_X g_{u_{k,i}}^p \dmu + \frac{p}{h_k(q+1)}\int_X\bigl(  u_{k,i}^{q+1} - (q+1) u^{q}_{k,i-1} u_{k,i}\bigr)\dmu
    \leq \int_X g_{(1-h)u_{k,i}+h v}^p \dmu\\
    &\qquad\qquad+ \frac{p}{h_k(q+1)}\int_X\bigl( \lvert (1-h)u_{k,i}+h v\rvert^{q+1} - (q+1) u^{q}_{k,i-1}( (1-h)u_{k,i}+hv)\bigr)\dmu \\
    &\qquad \leq \int_X\bigl( (1-h)g_{u_{k,i}}^p+hg_v^p\bigr) \dmu\\
    &\qquad\qquad+ \frac{p}{h_k(q+1)}\int_X \bigl(\lvert (1-h)u_{k,i}+h v\rvert^{q+1} - (q+1)u^{q}_{k,i-1}( (1-h)u_{k,i}+hv)\bigr)\dmu.
\end{split}
\end{equation*}
Rearranging the terms gives
\begin{align*}
    &\int_X g_{u_{k,i}}^p \dmu + \frac{p}{h_k(q+1)} \int_X \Bigl(\frac{1}{h}( u_{k,i}^{q+1} - \lvert (1-h)u_{k,i}+h v\rvert^{q+1}) - (q+1) u^{q}_{k,i-1} ( u_{k,i}-v)\Bigr)\dmu \\
    &\qquad\leq \int_X g_v^p \dmu,
\end{align*}
and by letting $h\to 0$, we deduce that 
\begin{equation}
    \label{eq: tested var ineq for discretization}
    \int_X g_{u_{k,i}}^p \dmu + \frac{p}{h_k}\int_X (u_{k,i}^{q} - u_{k,i-1}^{q})(u_{k,i}-v)\dmu \leq \int_X g_v^p\dmu.
\end{equation}

Next, we will apply \eqref{eq: tested var ineq for discretization} to obtain estimates for $u_{k,i}$. Let $h>0$, $\ell\in\mathbb N$ and let $v=u_{k,i}-h\min\{\ell u_{k,i},1\}$. Testing \eqref{eq: tested var ineq for discretization} with $v$ gives
\begin{equation*}
\begin{split}
    &\int_X g_{u_{k,i}}^p \dmu + \frac{hp}{h_k}\int_X (u_{k,i}^{q} - u_{k,i-1}^{q}) \min\{\ell u_{k,i},1\}\dmu \\
    &\qquad\leq \int_{X\cap \{\ell u_{k,i}>1\}} g_u^p\dmu + \lvert 1-h\ell\rvert^p\int_{X\cap \{\ell u_{k,i}<1\}} g_u^p\dmu.
\end{split}
\end{equation*}
Rearranging terms we obtain
\begin{equation*}
    \frac{p}{h_k}\int_X (u_{k,i}^{q} - u_{k,i-1}^{q}) \min\{\ell u_{k,i},1\}\dmu \leq \frac{1}{h}(\lvert 1-h\ell\rvert^p-1)\int_{X\cap \{\ell u_{k,i}<1\}} g_u^p\dmu.
\end{equation*}
Letting $h\to 0$, we have
\begin{equation*}
    \lim_{h\to 0}\frac{1}{h}(\lvert 1-h\ell\rvert^p-1)\int_{X\cap \{\ell u_{k,i}<1\}} g_u^p\dmu = -p\ell^{p-1}\int_{X\cap \{\ell u_{k,i}<1\}} g_u^p\dmu \leq 0,
\end{equation*}
so that 
\begin{equation*}
    \frac{p}{h_k}\int_X (u_{k,i}^{q} - u_{k,i-1}^{q}) \min\{\ell u_{k,i},1\}\dmu\leq 0,
\end{equation*}
from which we conclude that
\begin{equation*}
    \int_X u_{k,i}^{q}\min\{\ell u_{k,i},1\}\dmu\leq \int_X u_{k,i-1}^{q}\dmu.
\end{equation*}
Letting $\ell\to\infty$, we have
\begin{equation*}
    \int_X u_{k,i}^{q}\dmu\leq \int_X u_{k,i-1}^{q}\dmu.
\end{equation*}
Applying the above recursively shows 
\begin{equation}
    \label{eq: discrete ineq for mass preserv}
    \int_{X} u_{k,i}^{q}\dmu \leq \int_{X} u_0^{q}\dmu.
\end{equation}

For $u,v\in \mathbb R$, we denote
\begin{equation*}
     b(u,v) = \lvert u\rvert^{q+1}-\rvert v\rvert^{q+1}-(q+1) v^{q}(u-v).
\end{equation*}
By \cite[Lemma 2.5]{BogeleinDuzaarMarcelliniScheven}, there exists a constant $c$, depending only on $q$, such that
\begin{equation}
    \label{eq: inequality for b}
    c\lvert u^\frac{q+1}{2}-v^\frac{q+1}{2}\rvert^2 \leq  b(u,v) \leq (u^{q}-v^{q})(u-v)
\end{equation} 
for every $u,v\geq 0$.
Applying \eqref{eq: tested var ineq for discretization} with $v=u_{k,i-1}$ and using \eqref{eq: inequality for b} yields
\begin{equation}
\label{eq: eq for differences and energy}
    \int_X g_{u_{k,i}}^p \dmu+ \frac{1}{h_k} \int_X b(u_{k,i},u_{k,i-1})\dmu \leq \int_X g_{u_{k,i-1}}^p\dmu.
\end{equation}
A recursive application of \eqref{eq: eq for differences and energy} gives
\begin{equation}
    \label{eq: iterated est for difference quotients}
    \sum_{i=1}^k \frac{1}{h_k}\int_X  b(u_{k,i},u_{k,i-1})\dmu \leq \int_X g_{u_0}^p\dmu.
\end{equation}
Moreover, we conclude from \eqref{eq: eq for differences and energy} that 
\begin{equation}
    \label{eq: discrete energy decrease}
    \int_X g_{u_{k,i}}^p\dmu \leq \int_X g_{u_0}^p\dmu,
\end{equation}
for every $i\in \{0,\dots, k\} $. Applying \eqref{eq: tested var ineq for discretization} with $v=0$ and using Young's inequality shows that
\begin{equation*}
\begin{split}
    \int_X g_{u_{k,i}}^p \dmu + \frac{p}{h_k} \int_X   u_{k,i} ^{q+1} \dmu & \leq \frac{p}{h_k} \int_X  u_{k,i-1}^{q} u_{k,i}\dmu \\
    & \leq \frac{p}{h_k} \int_X \bigl(\frac{q}{q+1} u_{k,i-1}^{q+1} + \frac{1}{q+1} u_{k,i}^{q+1}\bigr)\dmu.
\end{split}
\end{equation*}
Absorbing the first term on the right hand side of the above it follows that
\begin{equation*}
    \int_X g_{u_{k,i}}^p \dmu + \frac{pq}{h_k(q+1)} \int_X   u_{k,i} ^{q+1} \dmu \leq \frac{p q}{h_k(q+1)} \int_X u_{k,i-1}^{q+1}\dmu.
\end{equation*}
Iterating the previous estimate finally implies that
\begin{equation}
    \label{eq: bound for gradient and sup in exist}
    \sum_{i=1}^k h_k \int_X g_{u_k,i}^p\dmu + \max_{i\in \{1,\dots,k\}} \int_X  u_{k,i}^{q+1}\dmu\leq c\int_X  u_0^{q+1}\dmu.
\end{equation}

Let $u_k:X\times(-h_k,T]\to\mathbb R$, $u_{k}(\cdot,t)=u_{k,i}$, $(i-1)h_k<t\le ih_k$, $i\in \{0,\dots, k\}$.
From \eqref{eq: bound for gradient and sup in exist} we obtain 
\begin{equation*}
    \sup_{k\in \mathbb N} \Bigl(\sup_{t\in [0,T]} \int_X  u_k^{q+1}\dmu+ \int_{X_T} g_{u_k}^p\dnu\Bigr) < \infty,
\end{equation*}
and hence the sequence $(u_k)$ is bounded in $L^p(0,T;N^{1,p}(X))\cap L^\infty(0,T;L^{q+1}(X))$. By reflexivity of $N^{1,p}(X)$, there exists $u\in L^p(0,T;N^{1,p}(X))\cap L^\infty(0,T;L^{q+1}(X))$ such that, after passing to a subsequence if necessary, again denoted by $(u_k)$, we have $u_k\rightharpoonup u$ weakly in $L^p(0,T;N^{1,p}(X))$ as $k\to\infty$. 
We claim that $u_k^{q}\rightharpoonup u^{q}$ weakly in $L^{\frac{q+1}{q}}(X_T)$ as $k\to\infty$,
after passing to a further subsequence if necessary. This is immediate for $q=1$. For $q\neq 1$, we conclude that there exists $w\in L^{\frac{q+1}{q}}(X_T)$ such that $u_{k}^{q}\rightharpoonup w$ weakly in $L^{\frac{q+1}{q}}(X_T)$. 

We claim that $w=u^{q}$. To see this, let us verify that verify that the assumptions for Simon's compactness result \cite[Theorem 1]{Simon1987} are satisfied. First, \eqref{eq: iterated est for difference quotients} and \eqref{eq: inequality for b} imply
\begin{equation}
    \label{eq: primitive continuity estimate for existence}
    \sum_{i=1}^k \frac{1}{h_k}\int_X \mathfrak \lvert u_k(t_i)^\frac{q+1}{2}-u_k(t_{i-1})^\frac{q+1}{2}\rvert^2\dmu \leq M.
\end{equation}
Let $h>0$. If $h\le h_k$, then 
\begin{equation*}
\begin{split}
    &\int_0^{T-h} \| u_k(t+h)^\frac{q+1}{2}-u_k(t)^\frac{q+1}{2}\|_{L^{2}(X)}^2 \dt \\
    &\qquad \leq \sum_{i=1}^{k} \int_{t_i-h}^{t_i} \| u_k(t_{i})^\frac{q+1}{2}-u_k(t_{i-1})^\frac{q+1}{2}\|_{L^{2}(X)}^2 
    \leq h h_k M.
\end{split}
\end{equation*}
Then assume that $h>h_k$. Let $m\in\mathbb N$ be such that $m h_k< h < (m+1)h_k $. For $0<t<T-h$, Hölder's inequality yields
\begin{equation*}
\begin{split}
    \| u_k(t+h)^\frac{q+1}{2}-&u_k(t)^\frac{q+1}{2} \|_{L^2(X)}
    \leq \sum_{j=1}^{m} \| u_k(t+jh_k)^\frac{q+1}{2}-u_k(t+(j-1)h_k)^\frac{q+1}{2} \|_{L^2(X)} \\
    &\qquad+\| u_k(t+h)^\frac{q+1}{2}- u_k(t+m h_k)^\frac{q+1}{2}\|_{L^2(X)} \\
    &\leq  m^\frac{1}{2}\Bigl(\sum_{j=1}^m \| u_k(t+jh_k)^\frac{q+1}{2}-u_k(t+(j-1)h_k)^\frac{q+1}{2} \|_{L^2(X)}^2\Bigr)^\frac{1}{2}\\
    &\qquad+ \| u_k(t+h)^\frac{q+1}{2}- u_k(t+m h_k)^\frac{q+1}{2}\|_{L^2(X)}.
\end{split}  
\end{equation*}
We integrate the estimate above over $0<t<T-h$ and apply $m<\frac{h}{h_k}$ to have
\begin{equation*}
\begin{split}
    &\int_0^{T-h} \| u_k(t+h)^\frac{q+1}{2}-u_k(t)^\frac{q+1}{2}\|_{L^{2}(X)}^2 \dt\\ &\qquad\leq \frac{c h}{h_k} \int_{0}^{T-h} \sum_{j=1}^m\| u_k(t+j h_k)^\frac{q+1}{2}-u_k(t+(j-1)h_k)^\frac{q+1}{2}\|_{L^{2}(X)}^2\dt \\ 
    & \qquad\qquad + c\int_0^{T-h} \| u_k(t+h)^\frac{q+1}{2}- u_k(t+m h_k)^\frac{q+1}{2}\|_{L^2(X)}\dt \\
    & \qquad\leq cTM h + c T M h_k\leq c h. 
\end{split}
\end{equation*}
Hence, we conclude that 
\begin{equation}
    \label{eq: conv of shifts}
    \|u_{k}(t+h)^\frac{q+1}{2}-u_k(t)^\frac{q+1}{2}\|_{L^2(0,T-h;L^2(X))}\to 0
\end{equation}
as $h\to0$ uniformly in $k$. 

Let $U_i$, $i\in I$ be a set as in \Cref{def: rellich kondrachov property}. By the compactness of the embedding $N^{1,p}(X)\hookrightarrow L^p(U_i)$, \eqref{eq: discrete energy decrease}, \eqref{eq: bound for gradient and sup in exist} and the uniform boundedness of $u_k$, we conclude that there exists a compact subset $S$ of $L^p(U_i)$, with $u_{k,i}\in S$ for every $k\in\mathbb N$, $i\in \{0,\dots,k\}$. Let $\Tilde U_i$ be a subset of $U_i$ with $\mu(\Tilde U_i)<\infty$. Then $u_{k,i}$, $k\in\mathbb N$, $i\in \{0,\dots,k\}$ is a precompact in $L^p(\Tilde U_i)$. Since $\mu(\Tilde U_i)<\infty$ and $q+1\leq p$ we conclude that $u_{k,i}$, $k\in\mathbb N$, $i\in \{0,\dots,k\}$ is precompact in $L^{q+1}(\Tilde U_i)$. By the continuity of power maps between Lebesgue spaces we conclude that there exists a compact subset of $L^{2}(\Tilde U_i)$ containing $\int_{t_1}^{t_2}u_k(t)^\frac{q+1}{2}\dt$ for every $0<t_1,t_2<T$ and $k\in\mathbb N$. Thus, we use \eqref{eq: conv of shifts} and apply \cite[Theorem 1]{Simon1987} to conclude that, after passing to a subsequence if necessary, $u_k^\frac{q+1}{2}\to v$ strongly in $L^2(\Tilde U_i)$ as $k\to\infty$ for some $v\in L^2(\Tilde U_i)$. We conclude that $u_k\to v^{\frac{2}{q+1}}$ strongly in $L^{q+1}(\Tilde U_i)$ as $k\to\infty$, which implies that $v=u^\frac{q+1}{2}$. Similarly, we conclude that $u_k^{q}\to u^q$ strongly in $L^\frac{q+1}{q}(\Tilde U_i)$, which implies that  $w=u^{q}$ in $\Tilde U_i$, and in turn $w=u^q$ in $U_i$. Since $i\in I$ was arbitrary and $X\subset \bigcup_{i\in I} U_i$, the claim follows.

It follows that $u_k^{q}\rightharpoonup u^{q}$ weakly in $L^1(B)$ as $k\to\infty$ for any ball $B=B(x_0,r)$ in $X$. From \eqref{eq: discrete ineq for mass preserv} we have
\begin{equation*}
    \int_X u_k(t)^{q}\dmu\leq \int_X u_0^{q}\dmu.
\end{equation*}
Let $\delta>0$. Integrating the above over $\tau<t<\tau+\delta$, we have
\begin{equation*}
    \frac{1}{\delta}\int_{X\times(\tau,\tau+\delta)} u_k^{q}\dnu \leq \int_X u_0^{q}\dmu.
\end{equation*}
Thus for any ball $B=B(x_0,r)$ we have
\begin{equation*}
    \frac{1}{\delta}\int_{B\times(\tau,\tau+\delta)} u^{q}\dnu\leq \liminf_{k\to\infty}\frac{1}{\delta}\int_{X\times(\tau,\tau+\delta)} u_k^{q}\dnu \leq \int_X u_0^{q}\dmu.
\end{equation*}
Since the ball $B$ was arbitrary, the above implies that
\begin{equation*}
    \frac{1}{\delta}\int_{X\times(\tau,\tau+\delta)} u^{q}\dnu\leq \int_X u_0^{q}\dmu.
\end{equation*}
Letting $\delta\to 0$, we obtain
\begin{equation*}
    \int_{X} u(\tau)^{q}\leq \int_X u_0^{q}\dmu,
\end{equation*}
which implies \eqref{eq: mass conserv}. Similarly, from \eqref{eq: discrete energy decrease} we have
\begin{equation*}
    \int_X g_{u_k}^p(t)\dmu\leq \int_{X} g_{u_0}^p\dmu,
\end{equation*}
and arguing similarly as before we conclude that \eqref{eq: neumann problem energy decrease} holds.

We have already seen that $u_k^\frac{q+1}{2}\rightharpoonup u^\frac{q+1}{2}$ weakly in $L^2(X_T)$ as $k\to\infty$. Again, this is immediate if $q=1$. We have also shown that 
\begin{equation*}
    w_k(t)=\frac{1}{h_k}\bigl(u_k(t+h_k)^\frac{q+1}{2}-u_k(t)^\frac{q+1}{2}\bigr),
\end{equation*}
is bounded in $L^2(X\times(0,T-h_k))$, independently of $k\in\mathbb N$. As such, we may pass to a further subsequence, still denoted by $w_k$, such that the sequence $(w_k)$ has a weak limit $w$ in $L^2(X_T)$. For every $\varphi\in C_0^\infty([0,T];L^p(X))$ we have
\begin{equation*}
\begin{split}
    \int_{X_T} u^\frac{q+1}{2} \partial_t \varphi \dnu & = -\lim_{k\to\infty} \int_{X_T} u_k(t)^\frac{q+1}{2}\frac{\varphi(t)-\varphi(t-h_k)}{h_k}\dnu \\
    & = -\lim_{k\to\infty} \int_{X_T} \frac{u_k(t+h_k)^\frac{q+1}{2}-u_k(t)^\frac{q+1}{2}}{h_k} \varphi(t)\dnu \\
    & = -\int_{X_T} w \varphi\dnu,
\end{split}
\end{equation*}
which implies that $\partial_t u^\frac{q+1}{2}=w\in L^2(X_T)$. This gives $u^\frac{q+1}{2}\in C([0,T];L^2(X))$ and hence $u\in C([0,T];L^{q+1}(X))$.

Let $v\in L^p(0,T;N^{1,p}(X))\cap L^{q+1}(X)$ with $\partial_t v\in L^{q+1}(X_T)$, and let $0<\tau<T$. Integrating \eqref{eq: tested var ineq for discretization} over $0<t<\tau$ gives
\begin{equation*}
    \int_{X_\tau} g_{u_{k}}^p \dnu + \frac{p}{h_k}\int_{X_\tau} (u_{k}^{q} - u_{k}(\cdot-h_k)^{q})(u_{k}-v)\dnu 
    \leq \int_{X_\tau} g_v^p\dnu.
\end{equation*}
Young's inequality shows that
\begin{equation*}
\begin{split}
    &\frac{p}{h_k}\int_{X_\tau}\bigl(u_k^{q}-u_k(\cdot-h_k)^{q}(-u_k)\bigr)\dnu\\
    &\qquad= \frac{p}{h_k}\int_{X_T}\bigl( u_{k}(\cdot-h_k)^{q} u_k -  u_k^{q+1}\bigr)\dnu \\
    &\qquad \leq \frac{p}{h_k}\int_{X_\tau}\Bigl(\frac{q}{q+1} u_k(\cdot-h_k)^{q+1} + \frac{1}{q+1}  u_k^{q+1} -  u_k^{q+1}\Bigr)\dnu \\ 
    &\qquad = \frac{p q}{h_k(q+1)} \int_{X_\tau} \bigl( u_k(\cdot-h_k)^{q+1}- u_k^{q+1}\bigr)\dnu \\
    &\qquad = \frac{pq}{q+1}\int_{X}  u_0^{q+1}\dmu - \frac{pq}{h_k(q+1)}\int_{X\times(\tau-h_k,\tau)} u_k^{q+1}\dnu.
\end{split}
\end{equation*}
We also have
\begin{equation*}
\begin{split}
    &\frac{p}{h_k}\int_{X_\tau}(u_k^{q}-u_k(\cdot-h_k)^{q})v \dnu\\
    &\qquad= \frac{p}{h_k}\int_{X_\tau} u_k^{q} v \dnu -\frac{p}{h_k}\int_{X_\tau} u_k(\cdot-h_k)^{q} v\dnu \\
    &\qquad = \frac{p}{h_k} \int_{X_\tau} u_k^{q} v \dnu - \frac{p}{h_k}\int_{X\times(-h_k,\tau-h_k)} u_k^{q} v(\cdot+h_k)\dnu\\
    &\qquad = - \frac{p}{h_k}\int_{X_\tau}  (v(\cdot+h_k)-v) u_k^{q}\dnu-\frac{p}{h_k}\int_{X\times(-h_k,0)} u_0^{q}v(\cdot+h_k)\dnu \\
    &\qquad \qquad + \frac{p}{h_k}\int_{X\times(\tau-h_k,\tau)}u_k^{q}v(\cdot+h_k)\dnu.
\end{split}
\end{equation*}
Combining the previous estimates, we obtain 
\begin{equation}
\label{eq: est for limiting var ineq and initial val}
\begin{split}
    & \frac{pq}{h_k(q+1)}\int_{X\times(\tau-h_k,\tau)} u_k^{q+1}\dnu+\int_{X_\tau} g_{u_{k}}^p \dnu
    \leq \int_{X_\tau} g_v^p\dnu\\
    &\qquad-\frac{p}{h_k}\int_{X_\tau}  (v(\cdot+h_k)-v) u_k^{q}\dnu
    -\frac{p}{h_k}\int_{X\times(-h_k,0)} u_0^{q}v(\cdot+h_k)\dnu\\
    &\qquad+ \frac{p}{h_k}\int_{X\times(\tau-h_k,\tau)}u_k^{q}v(\cdot+h_k)\dnu + \frac{pq}{q+1}\int_{X}  u_0^{q+1}\dmu.
\end{split}
\end{equation}

Let us first use \eqref{eq: est for limiting var ineq and initial val} to conclude that $u$ has the correct initial value. Testing \eqref{eq: est for limiting var ineq and initial val} with $v=u_0$ gives
\begin{equation*}
\begin{split}
    & \frac{pq}{h_k(q+1)}\int_{X\times(\tau-h_k,\tau)} u_k^{q+1}\dnu+\int_{X_\tau} g_{u_{k}}^p \dnu
    \leq \tau \int_{X} g_{u_0}^p\dmu\\
    &\qquad- \frac{p}{(q+1)}\int_{X} u_0^{{q+1}}\dmu + \frac{p}{h_k}\int_{X\times(\tau-h_k,\tau)}u_k^{q}u_0\dnu.
\end{split}
\end{equation*}
Rearranging the terms and applying \eqref{eq: inequality for b} shows that
\begin{equation*}
\begin{split}
    \frac{p}{h_k(q+1)}\int_{X\times(\tau-h_k,\tau)} b(u_0,u_k)\dnu 
    &\leq  \frac{p}{h_k(q+1)}\int_{X\times(\tau-h_k,\tau)}\bigl( q u_k^{q+1}-(q+1)u_k^{q}u_0 + u_0^{q+1}\bigr)\dnu \\
    &\leq \tau\int_{X} g_{u_0}^p\dmu  .
\end{split}
\end{equation*}
Let $\delta>0$. Integrating the estimate above over $\tau<t<\tau+\delta$, and using  $ b(u_0,u_k)\geq 0$, we obtain
\begin{equation*}
    \frac{1}{\delta}\frac{1}{h_k}\int_\tau^{\tau+\delta}\int_{X\times(t-h_k,t)} b(u_0,u_k)\dmu\ds\dt \leq \frac{\tau (q+1)}{p}\int_{X} g_{u_0}^p\dmu.
\end{equation*}
For $h_k<\delta$, we have
\begin{equation*}
    \frac{1}{\delta}\int_\tau^{\tau+\delta-h_k} b(u_0,u_k)\dnu \leq \frac{\tau(q+1)}{p}\int_{X} g_{u_0}^p\dmu.
\end{equation*}
Letting $k\to\infty$ and then $\delta \to 0$ shows
\begin{equation*}
    \int_{X} b(u_0,u(\tau))\dmu \leq \frac{\tau(q+1)}{p}\int_{X} g_{u_0}^p\dmu
\end{equation*}
for every $0<\tau<T$. This implies that $ b(u_0,u(t))\to 0 $ as $t\to 0$, which shows that $u(t)\to u_0$ in $L^{q+1}(X)$ as $t\to 0$.

Next we apply \eqref{eq: est for limiting var ineq and initial val} to show that \eqref{eq: var ineq for Neumann} holds. Let $\varphi$ be as in \eqref{eq: var ineq for Neumann}. Let $v\in L^p(0,T;N^{1,p}(X))\cap L^{q+1}(X)$ with $\partial_t v\in L^{q+1}(X_T)$ such that $v$ is compactly supported in $(0,\tau)$. Applying $v$ as a test function in  \eqref{eq: est for limiting var ineq and initial val} gives
\begin{equation}
\label{eq: var ineq before averaging and limit}
\begin{split}
    & \frac{pq}{h_k(q+1)}\int_{X\times(\tau-h_k,\tau)} u_k^{q+1}\dnu + \int_{X_\tau} g_{u_{k}}^p \dnu
    \leq \int_{X_\tau} g_v^p\dnu\\
    &\qquad+ \frac{pq}{q+1}\int_{X} u_0^{q+1}\dmu - \frac{p}{h_k}\int_{X_\tau}  (v(\cdot+h_k)-v) u_k^{q}\dnu,
\end{split}
\end{equation}
whenever $h_k$ is small enough. Let $\delta >0$, and integrate the above over $t_0<\tau<t_0+\delta$ to obtain 
\begin{equation*}
\begin{split}
    & \frac{pq}{\delta(q+1)}\int_{X\times(t_0,t_0+\delta-h_k)} u_k^{q+1}\dnu + \int_{X_\tau} g_{u_{k}}^p \dnu
    \leq \int_{X_{\tau+\delta}} g_v^p\dnu\\
    &\qquad+ \frac{pq}{q+1}\int_{X}u_0^{q+1}\dmu - \frac{1}{\delta}\int_{t_0}^{t_0+\delta}\frac{p}{h_k}\int_{X_\tau}  (v(\cdot+h_k)-v) u_k^{q}\dmu\dt\,d\tau,
\end{split}
\end{equation*}
for $h_k<\delta$. Letting $k\to \infty$ the above becomes
\begin{equation*}
\begin{split}
    & \frac{pq}{\delta(q+1)}\int_{X\times(t_0,t_0+\delta)} u^{q+1}\dnu + \int_{X_\tau} g_{u}^p \dnu 
    \leq \int_{X_{\tau+\delta}} g_v^p\dnu\\
    &\qquad+ \frac{pq}{q+1}\int_{X} u_0^{q+1}\dmu - \frac{1}{\delta}\int_{t_0}^{t_0+\delta}\frac{p}{h_k}\int_{X_\tau}  \partial_t v u^{q}\dmu\dt\,d\tau,
\end{split}
\end{equation*}
where we used the weak convergence of $u_k^{q}\rightharpoonup u^{q}$ in $L^{\frac{q+1}{q}}(X_T)$ as $k\to\infty$. Letting $\delta \to 0$ and taking $t_0=\tau$, this shows that
\begin{equation}
\label{eq: exist var ineq before time moll}
\begin{split}
    & \frac{pq}{q+1}\int_{X}u(\tau)^{q+1}\dnu + \int_{X_\tau} g_{u}^p \dnu \\
    & \qquad\leq \int_{X_\tau} g_v^p\dnu  + \frac{pq}{q+1}\int_{X}u_0^{q+1}\dmu - p \int_{X_\tau} \partial_t v u^{q}\dnu.
\end{split}
\end{equation}
Let $\delta>0$, $\epsilon>0$, and define a cutoff function $\eta$ by
\begin{equation*}
    \eta(t)
        =\begin{cases}
            0 ,&\quad t\leq \epsilon,\\
            \frac{1}{\epsilon}(t-2\epsilon), & \quad \epsilon <t\leq 2\epsilon, \\
            1,&\quad 2\epsilon < t \leq \tau-2\epsilon,\\
            \frac{1}{\epsilon}(\tau-2\epsilon-t), & \quad \tau-2\epsilon<t\leq \tau-\epsilon,\\
            0,&\quad \tau-\epsilon<t.
        \end{cases}
\end{equation*}
Testing \eqref{eq: exist var ineq before time moll} with $v=\eta (u_\delta +\varphi)$, we obtain
\begin{equation}
\label{eq: converting between forms of weak sol}
\begin{split}
    & \frac{pq}{q+1}\int_{X}u(\tau)^{q+1}\dmu + \int_{X_\tau} g_{u}^p \dnu
    \leq \int_{X_\tau} g_{ \eta u_\delta + \varphi}^p\dnu\\
    &\qquad + \frac{pq}{q+1}\int_{X} u_0^{q+1}\dmu - p \int_{X_\tau} \partial_t (\eta u_\delta +\varphi) u^{q}\dnu.
\end{split}
\end{equation}
For the term containing the time derivative, we use integration by parts and \eqref{eq: time der of mollification} to infer that
\begin{equation*}
\begin{split}
    &- p \int_{X_\tau} \partial_t (\eta u_\delta +\varphi) u^{q}\dnu  = -p \int_{X_\tau} \partial_t \varphi u^{q}\dnu-p \int_{X_\tau} \partial_t \eta u_\delta  u^{q}\dnu \\
    & \qquad\qquad-p \int_{X_\tau} \eta \partial_t u_\delta  (u^{q}-(u_\delta)^{q})\dnu -p \int_{X_\tau} \eta \partial_t u_\delta (u_{\delta})^{q}\dnu \\
    &\qquad  \leq -p \int_{X_\tau} \partial_t \varphi u^{q}\dnu-p \int_{X_\tau} \partial_t \eta u_\delta  u^{q}\dnu 
     + \frac{p}{q+1} \int_{X_\tau} \partial_t\eta (u_\delta)^{q+1} \dnu.
\end{split}
\end{equation*}
We use the estimate above and \Cref{lemma: convergence of time mollification} to pass $\delta\to 0$ in \eqref{eq: converting between forms of weak sol} and conclude that
\begin{equation}
\label{eq: exist var ineq before eps to 0}
\begin{split}
    & \frac{pq}{q+1}\int_{X}u(\tau)^{q+1}\dmu + p\int_{X_T} \partial_t\varphi u^{q}\dnu + \int_{X_\tau} g_{u}^p \dnu \\
    & \qquad\leq \int_{X_T} g_{ \eta u + \varphi}^p\dnu  + \frac{pq}{q+1}\int_{X} u_0^{q+1}\dmu - \frac{pq}{q+1} \int_{X_T} \partial_t \eta u^{q+1}\dnu.
\end{split}
\end{equation}
The definition of $\eta$ gives
\begin{equation*}
    \int_{X_T} \partial_t \eta u^{q+1}\dnu= \frac{1}{\epsilon} \int_{X\times(\epsilon,2\epsilon)} u^{q+1}\dnu - \frac{1}{\epsilon} \int_{X\times(\tau-\epsilon,\tau-2\epsilon)} u^{q+1}\dnu.
\end{equation*}
Hence, letting $\epsilon \to 0$ in \eqref{eq: exist var ineq before eps to 0} yields
\begin{equation*}
\begin{split}
    p\int_{X_T} \partial_t\varphi u^{q+1}\dnu + \int_{X_\tau} g_{u}^p \dnu \leq \int_{X_T} g_{ \eta u + \varphi}^p\dnu,
\end{split}
\end{equation*}
which shows that \eqref{eq: var ineq for Neumann} holds. 
\end{proof}

\section{Intrinsic parabolic Harnack inequality implies Poincaré inequality}
\label{section: PHI implies Poincare}
This section will finish the proof of \Cref{thm: main thm}, by showing that the intrinsic parabolic Harnack inequality implies a $(1,p)$-Poincaré inequality.
\begin{definition}
\label{def: PHI}
Let $1<p<\infty$ and $0<q<\infty$. The intrinsic parabolic Harnack inequality holds in $(X,d,\mu)$ if there exist constants $c_1>0$ and $c_2>0$ such that whenever $u$ is a nonnegative solution of \eqref{eq: PDE} in $\Omega_T$ with $\Omega\subset X$ and $T>0$, then for almost every $(x_0,t_0)\in\Omega_T$, if $u(x_0,t_0)>0$, $r>0$ and
$B(x_0,4r)\times (t_0- c_1 u(x_0,t_0)^{q+1-p} (4r)^p,t_0+c_1u(x_0,t_0)^{q+1-p}(4r)^p)\subset \Omega_T$,
we have
\begin{equation*}
    u(x_0,t_0) \leq c_2\essinf_{Q^+} u,
\end{equation*}
where $Q^+=B(x_0,r)\times (t_0+c_1u(x_0,t_0)^{q+1-p}/2 r^p,t_0+ c_1 u(x_0,t_0)^{q+1-p} r^p)$.
\end{definition}

In the following proposition, we obtain a Poincaré-type inequality for functions with zero boundary values from the intrinsic parabolic Harnack inequality.
\begin{proposition}
\label{prop: mean value ineq implies local sobolev}
Let $1<p<\infty$ and $0<q<\infty$ with $q+1<p$. Assume that the parabolic Harnack inequality holds, $N^{1,p}(X)$ is reflexive, and for $q\neq 1$ assume that the Rellich--Kondrachov property in \Cref{def: rellich kondrachov property} holds. Then there exists constants $c$ and $\gamma>0$, depending on $c_1$, $c_2$, $d$, $p$ and $q$, such that 
\begin{equation*}
    \left(\dashint_{B(x_0,r)} \lvert f\rvert^{q+1}\dmu \right)^\frac{1}{q+1}  \leq c r \left(\dashint_{B(x_0,r)} g_f^p\dmu\right)^\frac{1}{p},
\end{equation*}
for every $f\in N^{1,p}_0(B(x_0,r))$, where $x_0\in X$ and $0<r\leq \gamma\diam(X)$.
\end{proposition}
\begin{proof}
Let $f\in N^{1,p}_0(B(x_0,r))$. We may assume that $f$ is nonnegative and bounded. Let $u$ be the nonnegative solution of the Cauchy problem \eqref{eq: Cauchy problem}, given by \Cref{thm: Cauchy existence} with $u_0=f$ and $T>0$ is to be chosen later. Let $\tau>0$. For $h>0$, let the function $\zeta_h$ be given by
\begin{equation*}
    \zeta_h(t)
        =\begin{cases}
            \frac{1}{h}(t-h) ,&\quad 0<t < h,\\
            1,&\quad h\leq t\leq \tau,\\
            \frac{1}{h}(\tau+h-t),&\quad \tau<t\leq \tau+h,\\
            0,&\quad\text{otherwise}.
        \end{cases}
\end{equation*}
Applying \Cref{lemma: var sol implies weak sol with inequality} with $\varphi= \zeta_h f$ we have
\begin{equation*}
\begin{split}
    & \frac{1}{h}\int_{0}^{h}\int_{B(x_0,r)} u^{q} f\dmu\dt \\
    & \leq  \int_{B(x_0,r)\times(0,\tau)} g_u^{p-1} g_f\dmu\dt + \frac{1}{h}\int_{\tau}^{\tau+h}\int_{B(x_0,r)}  u^{q}f\dmu\dt.
\end{split}
\end{equation*}
Letting $h\to 0$, the above becomes
\begin{equation*}
    \int_{B(x_0,r)} f^{q+1}\dmu \leq  \int_{B(x_0,r)\times(0,\tau)} g_u^{p-1}g_f\dmu\dt + \int_{B(x_0,r)}  u(\tau)^{q}f\dmu.
\end{equation*}
Young's inequality implies 
\begin{equation*}
    \int_{B(x_0,r)} f^{q+1}\dmu \leq c\int_{B(x_0,r)\times(0,\tau)} g_u^{p-1}g_f\dmu\dt + \int_{B(x_0,r)}  u(\tau)^{q+1}\dmu.
\end{equation*}
Applying Hölder's inequality and \eqref{eq: neumann problem energy decrease} shows that
\begin{equation*}
    \dashint_{B(x_0,r)} f^{q+1}\dmu \leq c\tau\dashint_{B(x_0,r)} g_f^{p}\dmu + \dashint_{B(x_0,r)}  u(\tau)^{q+1}\dmu.
\end{equation*}
Let $\sigma >1$, and denote 
\begin{equation*}
    I = \left( \dashint_{B(x_0,r)}f^{q+1}\dmu\right)^\frac{q+1-p}{q+1}.
\end{equation*}
Writing the previous estimate for $\tau=c_1 4^{p-q-1}(8\sigma r)^p I$, we have
\begin{equation}
    \label{eq: basic estimate for sobolev}
    \dashint_{B(x_0,r)} f^{q+1}\dmu \leq c\sigma^p r^p I\dashint_{B(x_0,r)} g_f^{p}\dmu + \dashint_{B(x_0,r)}  u(\tau)^{q+1}\dmu.
\end{equation}
Here, $c_1$ is as in \Cref{def: PHI}. Let us estimate the second term on the right hand side of the above. We claim that 
\begin{equation*}
    \dashint_{B(x_0,r)}  u(\tau)^{q+1}\dmu \leq \frac{1}{2}\dashint_{B(x_0,r)} f^{q+1}\dmu,
\end{equation*}
provided that $\sigma$ is large enough. First, we simply estimate
\begin{equation*}
    \dashint_{B(x_0,r)}  u(\tau)^{q+1}\dmu \leq \esssup_{B(x_0,\sigma r)} u(\tau)^{q+1}.
\end{equation*}
If it holds that
\begin{equation*}
    \esssup_{B(x_0,\sigma r)} u(\tau)^{q+1}\leq \frac{1}{2}\dashint_{B(x_0,r)} f^{q+1}\dmu,
\end{equation*}
we are done. Hence, we may assume that 
\begin{equation}
    \label{eq: diff assump in sobolev proof}
    \esssup_{B(x_0,\sigma r)} u(\tau)^{q+1}> \frac{1}{2}\dashint_{B(x_0,r)} f^{q+1}\dmu.
\end{equation}
Let $y\in B(x_0,\sigma r)$ such that $\esssup_{B(x_0,\sigma r)} u(\tau) \leq 2 u(y,\tau)$. Then, \eqref{eq: diff assump in sobolev proof} implies 
\begin{equation*}
    \theta= c_1 u(y,\tau)^{q+1-p} \leq c_1 4^{p-1-q} I.
\end{equation*}
Hence, by choice of $\tau$, $u$ is a weak solution in 
\begin{equation*}
    B(y,8\sigma r)\times(\tau-\theta (8\sigma r)^p,\tau+\theta(8\sigma r)^p),
\end{equation*}
and thus we may apply the intrinsic parabolic Harnack inequality to conclude that
\begin{equation}
\label{eq: application of mean value ineq}
    \esssup_{B(x_0,\sigma r)} u(\tau)^{q+1} \leq 2^{q+1} u(y,\tau)^{q+1}\leq c \essinf_{Q^+} u^{q+1}\leq c \left(\dashint_{Q^+} u^{q}\dmu\dt\right)^\frac{q+1}{q},
\end{equation}
where $Q_+=B(y,2\sigma r)\times(\tau+\theta/2 (8\sigma r)^p,\tau+\theta (8\sigma r)^p)$. For any $t>0$ we have 
\begin{equation*}
    \int_{B(x,2\sigma r)} u(t)^{q}\dmu\leq \int_{X} u(t)^q\dmu \leq  \int_{X} u(0)^q\dmu= \int_{B(x_0,r)} f^{q}\dmu,
\end{equation*}
where we applied \eqref{eq: mass conserv}. Hence \eqref{eq: application of mean value ineq} implies 
\begin{equation}
\begin{split}
    \dashint_{B(x_0,r)}  u(\tau)^{q+1}\dmu & \leq c \left(\frac{\mu(B(x_0,r)}{\mu(B(y,2\sigma r))}\right)^\frac{q+1}{q} \left(\dashint_{B(x_0,r)} f^{q}\dmu\right)^{\frac{q+1}{q}} \\
    & \leq c \left(\frac{\mu(B(x_0,r)}{\mu(B(x_0,\sigma r))}\right)^\frac{q+1}{q} \left(\dashint_{B(x_0,r)} f^{q}\dmu\right)^{\frac{q+1}{q}},
\end{split}
\end{equation}
where in the second step we used the fact that $B(x_0,\sigma r)\subset B(y,2\sigma r)$.
Hölder's inequality gives
\begin{equation*}
    \dashint_{B(x_0,r)}  u(\tau)^{q+1}\dmu \leq c \left( \frac{\mu(B(x_0,r)}{\mu(B(x_0,\sigma r))}\right)^\frac{q+1}{q}  \dashint_{B(x_0,r)} f^{q+1}\dmu\dt.
\end{equation*}
Since $\mu$ is doubling, \eqref{eq: doubling implies radii control from above} implies that there exists a constant $\alpha>0$, depending only on $c_\mu$, such that
\begin{equation*}
    \dashint_{B(x_0,r)}  u(\tau)^{q+1}\dmu\leq c \left( \frac{1}{\sigma}\right)^\frac{\alpha (q+1)}{q}  \dashint_{B(x_0,r)} f^{q+1}\dmu\dt,
\end{equation*}
provided that $\sigma r < \diam(X)/3$. Thus, choosing $\sigma $ large enough we have
\begin{equation*}
    \dashint_{B(x_0,r)} u(\tau)^{q+1}\leq \frac{1}{2} \dashint_{B(x_0,r)} f^{q+1}\dmu.
\end{equation*}
Returning to \eqref{eq: basic estimate for sobolev} we have shown
\begin{equation*}
    \dashint_{B(x_0,r)} f^{q+1}\dmu \leq cr^p I \dashint_{B(x_0,r)} g_f^p\dmu + \frac{1}{2} \dashint_{B(x_0,r)} f^{q+1}\dmu.
\end{equation*}
Absorbing the second term on the right hand side of the above implies
\begin{equation*}
    \dashint_{B(x_0,r)} f^{q+1}\dmu \leq cr^p I \dashint_{B(x_0,r)} g_f^p\dmu.
\end{equation*}
Recalling the definition of $I$ we deduce that
\begin{equation*}
    \left(\dashint_{B(x_0,r)}f^{q+1}\dmu \right)^\frac{p}{q+1}  \leq c r^p \dashint_{B(x_0,r)} g_f^p\dmu,
\end{equation*}
which finishes the proof.
\end{proof}
\begin{corollary}
\label{cor: PHI implies PI}
Let $1<p<\infty$ and $0<q<\infty$ with $q+1<p$. Assume that the parabolic Harnack inequality holds, $\mu$ is doubling and for $q\neq 1$ assume that the Rellich--Kondrachov property in \Cref{def: rellich kondrachov property} holds. Then a weak $(q+1,p)$-Poincaré inequality holds.
\end{corollary}
\begin{proof}
Combining the conclusion of \Cref{prop: mean value ineq implies local sobolev}, doubling, and the elliptic Harnack inequality we apply \cite[Theorem 7.6]{EleniusKinnunen}  which shows the claim. 
\end{proof}
We are now ready to prove \Cref{thm: main thm}.
\begin{proof}[Proof of \Cref{thm: main thm}]
The implication $(i)\implies(ii)$ is from \Cref{thm: PI implies PHI} and the implication $(ii)\implies(i)$ is from \Cref{cor: PHI implies PI}.
\end{proof}


\begin{thebibliography}{99}

\bibitem{AmbrosioColomboSimone}
L. Ambrosio, M. Colombo and S. Di Marino. Sobolev spaces in metric measure spaces: reflexivity and lower semicontinuity of slope, \emph{Variational methods for evolving objects, Adv. Stud. Pure Math.}, Vol 67, Mathematical Society of Japan, Tokyo (2015).

\bibitem{BarlowBassKumagai}
M.T. Barlow, R.F. Bass and T. Kumagai. Stability of parabolic {H}arnack inequalities on metric measure spaces, \emph{J. Math. Soc. Japan}, 58(2):485--519, 2006.


\bibitem{BjornBjorn2011}
A. Bj\"orn and J. Bj\"orn. {\it Nonlinear potential theory on metric spaces}, volume 17 of \emph{EMS Tracts in Mathematics}. European Mathematical Society (EMS), Z\"urich, 2011.

\bibitem{BogeleinDuzaarMarcellini}
V. B\"ogelein, F. Duzaar and P. Marcellini. Parabolic systems with {$p,q$}-growth: a variational approach, \emph{Arch. Ration. Mech. Anal.}, 210(1):219--267, 2013.

\bibitem{BogeleinDuzaarMarcelliniScheven}
V. B\"ogelein, F. Duzaar, P. Marcellini and C. Scheven. Doubly nonlinear equations of porous medium type, \emph{Arch. Ration. Mech. Anal.}, 229(2):503--545, 2018.

\bibitem{BogeleinDuzaarLiao}
V. B\"ogelein, F. Duzaar and N. Liao. On the {H}\"older regularity of signed solutions to a doubly nonlinear equation, \emph{J. Funct. Anal.}, 281(9): Paper No. 109173, 58, 2021.

\bibitem{BogeleinHeranSchatzlerSinger}
V. B\"ogelein, A. Heran, L. Sch\"atzler and T. Singer. Harnack's inequality for doubly nonlinear equations of slow diffusion type, \emph{Calc. Var. Partial Differential Equations}, 60(6), 2021.


\bibitem{Delmotte}
T. Delmotte. Parabolic {H}arnack inequality and estimates of {M}arkov chains on graphs, \emph{Rev. Mat. Iberoamericana}, 15(1):181--232, 1999.

\bibitem{DiBenedetto1993}
E. DiBenedetto. \emph{Degenerate parabolic equations}, Springer-Verlag, New York, 1993.

\bibitem{DiBenedettoGianazzaVespri2008}
E. DiBenedetto, U. Gianazza and V. Vespri. Harnack estimates for quasi-linear degenerate parabolic differential equations, \emph{Acta Math.}, 200(2):181--209, 2008.

\bibitem{DiBenedettoGIanazzaVespri2012}
E. DiBenedetto, U. Gianazza and V. Vespri. {\it Harnack's inequality for degenerate and singular parabolic equations}. Springer Monographs in Mathematics. Springer, New York, 2012.

\bibitem{EleniusKinnunen}
T. Elenius and J. Kinnunen. Doubling measures, Poincaré inequalities and parabolic Harnack inequalities for a doubly nonlinear equation, preprint, arXiv 2605.27665, 2026.


\bibitem{Grigoryan1991}
A. Grigor{'}yan. The heat equation on noncompact Riemannian manifolds, \emph{Mat. Sb.}, 182(1):55--87, 1991.

\bibitem{GrigoryanSurig2024}
A. Grigor{'}yan and P. S\"urig. Finite propagation speed for Leibenson's equation on Riemannian manifolds, \emph{Comm. Anal. Geom.}, 32(9):2467-2504, 2024.

\bibitem{GrigoryanSurig2024Pisa}
A. Grigor{'}yan and P. S\"urig. Sharp propagation rate for Leibenson's equation on Riemannian manifolds, \emph{Ann. Sc.}, to appear.

\bibitem{GrigoryanSurig2025}
A. Grigor{'}yan and P. S\"urig. Upper bounds for solutions of Leibenson's equation on Riemannian manifolds, \emph{J. Funct. Anal.},  288(10): Paper No. 110878, 30, 2025.

\bibitem{GrigoryanSunSurig}
A. Grigor{'}yan, J. Sun and P. S\"urig. Sharp long distance upper bounds for solutions of Leibenson's equation on Riemannian manifolds, preprint, arXiv 2603.27791, 2026.

\bibitem{HeinonenKoskelaShanmugalingam}
J. Heinonen, P. Koskela, N. Shanmugalingam and J. T. Tyson. {\it Sobolev spaces on metric measure spaces}, volume 27 of \emph{New Mathematical Monographs}. Cambridge University Press, Cambridge, 2015. 

\bibitem{Heran}
A. Her\'an. Harnack inequality for parabolic quasi minimizers on metric spaces, \emph{Atti Accad. Naz. Lincei Rend. Lincei Mat. Appl.}, 31(3):565-592, 2020.

\bibitem{KeithZhong}
S. Keith and X. Zhong. The {P}oincar\'e{} inequality is an open ended condition, \emph{Ann. of Math. (2)}, 167(2):575--599, 2008.

\bibitem{KinnunenShanmugalingam}
J. Kinnunen and N. Shanmugalingam. Regularity of quasi-minimizers on metric spaces, \emph{Manuscripta Math.}, 105(3):401--423, 2001.

\bibitem{Liao2021}
N. Liao. Regularity of weak supersolutions to elliptic and parabolic equations: lower semicontinuity and pointwise behavior, \emph{J. Math. Pures Appl. (9)}, 147:179--204, 2021.


\bibitem{Saloff-Coste1992}
L. Saloff-Coste. A note on {P}oincar\'e, {S}obolev, and {H}arnack inequalities, \emph{Internat. Math. Res. Notices}, (2):27--38, 1992.

\bibitem{Saloff-Coste2002}
L. Saloff-Coste. {\it Aspects of {S}obolev-type inequalities}, volume 289 of \emph{London Mathematical Society Lecture Note Series}. Cambridge University Press, Cambridge, 2002.

\bibitem{Simon1987}
J. Simon. Compact sets in the space {$L^p(0,T;B)$}, \emph{Ann. Mat. Pura Appl. (4)}, 146:65--96, 1987.

\bibitem{Sturm}
K.T. Sturm. Analysis on local {D}irichlet spaces. {III}. {T}he parabolic {H}arnack inequality, \emph{J. Math. Pures Appl. (9)}, 75(3):273--297, 1996.

\bibitem{sürig2026existence}
P. S{\"u}rig. Existence results for Leibenson's equation on Riemannian manifolds, preprint, arXiv 2601.20640, 2026.

\bibitem{sürig2026upperbounds}
P. S{\"u}rig. Long time upper bounds for solutions of Leibenson's equation on Riemannian manifolds, preprint, arXiv  2604.23227, 2026.

\bibitem{Urbano}
J. M. Urbano. \emph{The method of intrinsic scaling. A systematic approach to regularity for degenerate and singular PDEs}, volume 1930 of Lecture Notes in Mathematics. Springer-Verlag, Berlin, 2008.

\end{thebibliography}

\end{document}